\documentclass[12pt,centertags,oneside]{amsart}
\usepackage{amstext,amsthm,amscd,typearea,hyperref}
\usepackage{amsmath,amssymb}
\usepackage{a4wide}
\usepackage[mathscr]{eucal}
\usepackage{mathrsfs}
\usepackage{typearea}
\usepackage{charter}
\usepackage{pdfsync}
\usepackage{xcolor}
\usepackage[a4paper,width=16.2cm,top=3cm,bottom=3cm]{geometry}

\numberwithin{equation}{section}
\usepackage{hyperref}

\newtheorem{theorem}{Theorem}[section]

\newtheorem{proposition}[theorem]{Proposition}
\newtheorem{corollary}[theorem]{Corollary}
\newtheorem{lemma}[theorem]{Lemma}
\newtheorem{remark}[theorem]{Remark}

\newcommand{\Leb}{{\rm Leb}}

\newcommand{\supp}{{\rm supp}}

\newcommand{\dc}{d^c}

\def\d{\operatorname{d}}

\newcommand{\PSH}{{\rm PSH}}

\newcommand{\vol}{\mathop{\mathrm{Vol}}\nolimits}
\newcommand{\amp}{\mathop{\mathrm{Amp}}\nolimits}
\newcommand{\capa}{\mathop{\mathrm{Cap}}\nolimits}
\newcommand{\bmo}{\mathop{\mathrm{BMO}}\nolimits}
\newcommand{\bigclass}{\mathop{\mathrm{big}}\nolimits}

\newcommand{\C}{\mathbb{C}}

\newcommand{\N}{\mathbb{N}}

\newcommand{\R}{\mathbb{R}}

\title[]{\quad Non-collapsing volume estimate for local K\"{a}hler metrics in big cohomology classes}
\author{Thai Duong Do, Duc-Bao Nguyen and Duc-Viet Vu}
\newcommand{\Addresses}{
{\bigskip
\footnotesize

		\textsc{Thai Duong Do, Institute for Artificial Intelligence, University of Engineering and Technology, Vietnam National University, Hanoi, Vietnam.}
        
        \textsc{National University of Singapore, Department of Mathematics, 10 Lower Kent Ridge Road, 119076, Singapore.}
		\noindent
		\par\nopagebreak
		\noindent
		\textit{E-mail address}: \texttt{dtduong@vnu.edu.vn}, \texttt{dtduong@nus.edu.sg}}
{
		\bigskip
  \footnotesize
		
  \textsc{Duc-Bao Nguyen, National University of Singapore, Department of Mathematics, 10 Lower Kent Ridge Road, 119076, Singapore.}
		\noindent
		\par\nopagebreak
		\noindent
		\textit{E-mail address}: \texttt{ducbao.nguyen@u.nus.edu}}

{
		\bigskip
		\footnotesize
		\textsc{Duc-Viet Vu, University of Cologne, Division of Mathematics, Department of Mathematics and Computer Science, Weyertal 86-90, 50931, K\"oln.}
		\noindent
		\par\nopagebreak
		\noindent
		\textit{E-mail address}: \texttt{dvu@uni-koeln.de}
}}

\begin{document}

%\hyphenpenalty=10000

\date{\today}

\begin{abstract}
We prove a uniform local non-collapsing volume estimate for a large family of singular metrics in the big cohomology classes, which are K\"ahler on an open Euclidean subset of the manifold. The key ingredient is a generalization of a mixed energy estimate for functions in the complex Sobolev space to the setting of big cohomology classes.
  \end{abstract}

\medskip

\maketitle

\noindent {\bf Classification AMS 2020:} 32U15, 32Q15, 53C23.

\smallskip


\section{Introduction}

 There has been growing interest in studying the Gromov-Hausdorff convergence of K\"ahler spaces, in part motivated by questions concerning the degeneration of K\"ahler-Einstein metrics and the convergence of K\"ahler-Ricci flows; see \cite{Donaldson-Sun,DonaldsonSun2,Liu-Szekelyhidi,Liu-Szekelyhidi2,Tian-survey-KE,Tosatti-survey,Tosatti-KEflow,Song-Tian-canonicalmeasure}. A major breakthrough is the paper \cite{Guo-Phong-Song-Sturm} by Guo-Phong-Song-Sturm, which establishes a uniform diameter estimate and a non-collapsing volume estimate for K\"ahler metrics of Monge-Amp\`ere type. Main results in \cite{Guo-Phong-Song-Sturm} were refined in \cite{Vu-diameter,GuedjTo-diameter,GPSS_bodieukien}, extending to the semi-positive cohomology class and family settings, and in \cite{Nguyen_Vu_diamterforbigclass} to the big cohomology class case. Roughly speaking, if the Monge-Amp\`ere measure of the  (singular) K\"ahler metric $T$ (of minimal singularities) in $X$ satisfies a good enough global density property, then, by the previously cited works, we obtain uniform diameter estimates as well as non-collapsing volume estimates, thus providing a version of Gromov's pre-compactness theorem for the family of these metrics.

 In this paper, we are interested in a more general setting in which the regularity of Monge-Amp\`ere of $T$ is known only locally, e.g., in a given open subset of the Euclidean topology of $X$. Such a consideration has its roots in Perelman's work \cite{Perelman-entropy} and was considered in the context of K\"ahler geometry by Guo-Song \cite{Guo-Song-localnoncollapsing} in which they proved a local non-collapsing volume estimate by requiring only a lower bound for Ricci curvature in a small geodesic ball. This result was  subsequently extended in \cite{Vu-diameter} for metrics of bounded potentials in semi-positive cohomology classes  by reducing the problem to the case of metrics whose Monge-Amp\`ere are of bounded $L^1 \log L^p$ density. Such a reduction was achieved by utilizing the notion of complex Sobolev spaces and proving a mixed energy estimate for currents of bounded potentials. Our goal in this paper is to generalize this local non-collapsing volume estimate to metrics in big cohomology classes. Let us enter the details.

Let $(X,\omega_X)$ be a compact K\"{a}hler manifold of dimension $n$. We fix a norm $\|\cdot\|$ on the finite-dimensional vector space $H^{1,1}(X,\R)$. For every smooth form $\theta$ in a big cohomology class $\alpha$, we put $V_\theta:= \sup \{\varphi : \varphi \text{ $\theta$-psh}, \varphi \leq 0\}$. For each $\alpha \in H^{1,1}(X,\R)$, fix  a smooth closed form $\theta_\alpha \in \alpha $ such that $\|\theta_\alpha\|_{\mathscr{C}^2} \leq C \|\alpha\|$, where $C$ is a uniform constant independent of $\alpha$ (see, e.g, \cite{Guo-Phong-Song-Sturm} or \cite{Vu-diameter}). Let $A>0,K>0$ and $p_0>n$ be constants. Denote by $\mathcal{M}(X,A)$ the set of K\"ahler metrics $\omega$ such that $\omega = \theta_\alpha + dd^c \varphi_\alpha$ for some  $\theta_\alpha$-psh function $\varphi_\alpha$ with $\|\varphi_\alpha - V_{\theta_\alpha}\|_{L^\infty} \leq A$ and $\|\alpha\| \leq A$, where $\alpha$ is the cohomology class of $\omega$. 
For $x_0 \in X, R_0 \in (0,1]$, let $\mathcal{M}(X,A,x_0,R_0,p_0,K)$  be the set of $\omega \in \mathcal{M}(X,A)$ such that 
\[\int_{B_\omega (x_0,R_0)} f |\log f|^{p_0} \omega_X^n \leq K, \text{ where } f:= \frac{\omega^n}{V_\omega \cdot \omega_X^n}\cdot\]
 A local non-collapsing volume estimate for metrics in $\mathcal{M}(X,A,x_0,R_0,p_0,K)$ was obtained in \cite{Vu-diameter}. An advantage of this class is that one only has to control the local entropy of the metric $\omega$, whereas in other works, the global entropy is required.
To state our main results, we begin with some definitions.

Denote by $\mathcal{K}_{\bigclass}(X)$ the set of closed positive $(1,1)$-currents $T$ such that $T$ is a smooth K\"ahler form on an open (Euclidean) subset $U_T$ of $X$. Define $V_T:= \int_X \left \langle T^n \right \rangle$, where  $\left\langle \cdot \right \rangle$ is the non-pluripolar product in the sense of \cite{BEGZ}. It is clear that $V_T \geq \int_{U_T}T^n >0$. On $U_T$, the current $T$ induces a distance function $\d_T$, and we denote by $B_T(\cdot,\cdot)$ the ball with respect to this distance. For every Borel subset $E\subset U_T$, let $\vol_T(E)$ denote the volume of $E$ with respect to the metric $T$ on $U_T$. For $x \in U_T$, let $\d_T(x,\partial U_T)$ be the supremum of the set of radii $r>0$ such that the ball $B_{T}(x,r)$ is relatively compact in $U_T$.

Let $A> 0$ be a constant. We denote by $\mathcal{M}_{\bigclass}(X,A)$ the set of $T\in \mathcal{K}_{\bigclass}(X)$ such that $T = \theta_\alpha + dd^c \varphi_T$ for some $\theta_\alpha$-psh function $\varphi_T$ with $\|\varphi_T - V_{\theta_\alpha}\|_{L^\infty} \leq A$ and $\|\alpha\| \leq A$, where $\alpha$ is the cohomology class of $T$. Here is our main result.

\begin{theorem}\label{mainresult} Let $A > 0,K > 0$ and $p_0>1$ be constants. Let $T \in \mathcal{M}_{\bigclass}(X,A)$, $x_0 \in U_T$, and $R_0 \in (0, \d_T(x_0,\partial U_T)]$. Suppose that 
\begin{align}\label{ine-dkLplocal}
\int_{B_{T}(x_0,R_0)} |f|^{p_0} \omega_X^n \leq K, \text{ where } f:= V_T^{-1}   T^n /\omega_X^n \text{ is defined on }U_T.
\end{align}
  Then, for $q\in (1,n/(n-1))$, there exists a constant $ C= C(\omega_X,p_0,A,K,q)>0$ independent of $T,x_0$ and $R_0$ such that 
    \[\vol_T (B_{T}(x_0,r)) \geq C r^{4nq/(q-1)} V_T \text{ for every }r \in (0,R_0/2].\]
\end{theorem}

Note that, as in \cite{Nguyen_Vu_diamterforbigclass}, our results still hold if, in the assumption (\ref{ine-dkLplocal}), we replace the $L^{p_0}$-norm on $B_{T}(x_0,R_0)$ by the $L^1 \log L^{p}$-norm, provided that $n\geq 2$ and $p>n$.

We also obtain a similar corollary with \cite[Corollary 1.7]{Vu-diameter}. Let $U$ be an open (Euclidean) subset of $X$. We define $\mathcal{M}_{\bigclass}(X,A,p_0,K,U)$ to be the subclass of $\mathcal{M}_{\bigclass}(X,A)$ consisting of $T$ such that $U \Subset U_T$,
and 
\[\int_{U} |f|^{p_0} \omega_X^n \leq K,
\text{ where }f:= V_T^{-1}  T^n /\omega_X^n \text{ is defined on }U_T.\]

\begin{corollary}\label{corollaryofmainresult}
    For $q\in (1,n/(n-1))$, there exists a constant $ C= C(\omega_X,p_0,A,K,q)>0$ such that $\d_T(x,\partial U) \leq C^{-1}$ for every $x\in U$ and 
    \[ \vol_T (B_{T}(x,r)) \geq C r^{4nq/(q-1)} V_T\]
    for every $x \in U$, $r\in (0,\d_T(x,\partial U)/2]$, and $ T \in \mathcal{M}_{\bigclass}(X,A,p_0,K,U)$.
\end{corollary}

In order to prove Theorem \ref{mainresult}, we follow the same strategy as in the proof of \cite[Theorem 1.6]{Vu-diameter}. We first need an integration by-parts formula for functions in the complex Sobolev space. This has been done in \cite{Vu-diameter} provided that $T$ has bounded potentials. The difficulty arises when we deal with more general $T$, in that the non-pluripolar product is not always locally well-defined. To solve this issue, we establish a global regularization for functions in the complex Sobolev space, refining a previous result by Vigny \cite{Vigny} (see Theorem~\ref{regularize}). The integration by parts formula (Theorem 4.1) and the energy estimates (Section \ref{energy estimate section}) are proved using multiple regularization processes. The second ingredient is a Sobolev type inequality for metrics in big cohomology classes and  functions in complex Sobolev spaces (see Theorem \ref{sobolevW* 2}). We underline that the Sobolev inequality for metrics in big cohomology classes was known (\cite{Guo-Phong-Song-Sturm2,Nguyen_Vu_diamterforbigclass}), but under the assumption that the metrics are smooth on an open Zariski dense subset in $X$. We do not have such a condition in our present setting.  Theorem \ref{mainresult} will be obtained by arguing exactly as in the proof of \cite[Theorem 1.6]{Vu-diameter}, using energy estimates (Section \ref{energy estimate section}) and Theorem \ref{sobolevW* 2}.

We briefly discuss some potential applications and future developments of our results. As mentioned before, our results are motivated by the findings of Guo-Song in \cite{Guo-Song-localnoncollapsing}. The advantage of our results is that one only needs to check the $L^p$-norm of the entropy instead of trying to control the Ricci curvature of the metric in Guo-Song's paper. Another possible geometric application is that when we study the K\"ahler-Einstein metrics on quasi-projective varieties, e.g., on the complement of a divisor $D$ such that $K_X + D$ is big (see \cite{Tian-Yau-KE-complete, Berman-Guenancia, DangVu} for results in this direction), one can only control the $L^p$-norm of the entropy locally on $X\setminus D$. Finally, from the viewpoint of pluripotential theory, we note that the main technical part of our results is the analysis of the complex Sobolev spaces. Such analysis can be useful to characterize the H\"older continuous solution to the complex Monge-Amp\`ere equation in big cohomology classes, generalizing the main result of \cite{DKC_Holder-Sobolev} to the big cohomology setting.

The paper is organized as follows. In Section~\ref{analysisDS}, we prove a key global regularization theorem for functions in the complex Sobolev space. We generalize results for the complex Sobolev space in \cite{Vu-diameter} to currents of not necessarily bounded potentials in Sections~\ref{intbypart}, \ref{section CS}, and \ref{energy estimate section}. We prove a new Sobolev inequality for complex Sobolev functions in Section~\ref{section Sobolev ineq for complex Sobolev} and finish the proof of the main results in Section~\ref{proof of main result}. 

\hspace{1cm}

\noindent \textbf{Acknowledgment.} The first named author acknowledges the support of the Singapore Academies Southeast Asia Fellowship (SASEAF) Programme under the project “Math, Sobolev Space” (Grant No. E-146-00-0039-01). The second named author is supported by the Singapore International Graduate Award (SINGA). The third named author is partially supported by the Deutsche Forschungsgemeinschaft (DFG, German Research Foundation)-Projektnummer 500055552 and by the ANR-DFG grant QuaSiDy, grant no ANR-21-CE40-0016. Part of this work was carried out during a visit of the second named author at the Department of Mathematics and Computer Science, University of Cologne. He would like to thank them for their warm welcome and financial support. We would like to thank Tien-Cuong Dinh for many useful discussions about complex Sobolev spaces and for his support during the preparation of this work. Finally, we would like to thank the anonymous referee for many suggestions that greatly improved the paper.

\section{Regularization of functions in complex Sobolev spaces}\label{analysisDS}

Let $(X,\omega)$ be a compact K\"ahler manifold of dimension $n$. Let $W^{1,2}(X)$ be the Sobolev space of Borel measurable functions $u$ on $X$ such that $u$ and $\nabla u$ are square-integrable functions with respect to the measure $\omega^n$. We denote by $W^{1,2}_*(X)$ the set of functions $u$ in $W^{1,2}(X)$ such that there exists a closed positive $(1,1)$-current $T$ such that 
\begin{equation}\label{du wedge dc u <= T}du\wedge d^c u\leq T \text{ in the sense of currents.}\end{equation}
Here $d^c:= (i/2\pi) (\overline{\partial} - \partial)$. Recall that by \cite{Vigny}, $W^{1,2}_*(X)$ is a Banach space endowed with the norm 
\[\|u\|_*^2 := \int_X |u|^2 \omega^n + \inf \Big \{ \|T\|_X : = \int_X T\wedge \omega^{n-1} \Big \},\]
where the infimum is taken over all closed positive $(1,1)$-currents $T$ that satisfy \eqref{du wedge dc u <= T}.

This space is called \textit{the complex Sobolev space} and was introduced by Dinh-Sibony in \cite{DS_decay}. We refer the reader to \cite{Vigny} for further properties and \cite{DoNguyenW12} for a higher version of this space. We also refer to \cite{DLW,DLW2,Vigny_expo-decay-birational,Vu_nonkahler_topo_degree} for applications in complex dynamics, to \cite{DinhMarinescuVu,DKC_Holder-Sobolev} for applications in the complex Monge-Amp\`ere equation, and to \cite{Vu-log-diameter,Vu-diameter} for applications in complex geometry.

A function $u$ in $W^{1,2}_*(X)$ is a priori not well-defined pointwise. However, one can define good representatives of $u$ as follows (see also \cite{DinhMarinescuVu,Vigny-Vu-Lebesgue,Vu-diameter}). Let $U$ be a local chart of $X$. Let $\chi$ be a smooth radial cut-off function on $\C^n$ such that $\chi$ is compactly supported, $0\leq \chi \leq 1$, and $\int_{\C^n} \chi d \Leb = 1$, where $\Leb$ is the Lebesgue measure in $\C^n$. Define 
\[u_\varepsilon (z) := \varepsilon^{-2n} \int_{\mathbb{C}^n} u(z-x)\chi(x/\varepsilon) d\Leb,\]
for $\varepsilon \in (0,1]$. We call $u_\varepsilon$ the \textit{standard regularization} of $u$. A Borel function $u': X\rightarrow \mathbb{R}$ is said to be a \textit{good representative} of $u$ if, for every local chart $U$ and every standard regularization $(u_\varepsilon)_\varepsilon$ of $u$ on $U$, we have that $u_\varepsilon$ converges pointwise to $u'$ as $\varepsilon \rightarrow 0^+$ outside some pluripolar set in $U$. Good representatives always exist and differ from each other by a pluripolar set (see \cite[Theorem 2.10]{DinhMarinescuVu}). In this paper, we always consider good representatives of $u$.

 The following regularization theorem has been obtained by Vigny in \cite{Vigny}.

\begin{theorem}\cite[Theorem 10]{Vigny}\label{regularizeVig} Let $u$ be a function in $W^{1,2}_*(X)$.
Then there exists a sequence of smooth functions $(u_k)$ such that $u_k$ converges weakly as distributions to $u$. Moreover, there exists a constant $C$ that does not depend on $u$ and $k$, such that $\|u_k\|_* \leq C \|u\|_*$ for every $k$.
\end{theorem}

We will establish some properties of complex Sobolev functions in the next sections. A general strategy is to approximate using smooth functions. However, the weak convergence of this theorem is not enough for our purpose (see also the comment about the proof of Theorem~\ref{integrationbypartdsh}). Our goal in this section is to prove a stronger regularization theorem for functions in the complex Sobolev space, for which we will additionally obtain the pointwise convergence of the approximants to $u$.

Theorem~\ref{regularizeVig} has been proved in \cite{Vigny} using the method of Dinh-Sibony in \cite{DS_regula}.  We will further refine this method to obtain the following theorem.

\begin{theorem}\label{regularize}
     Let $u$ be a function in $W^{1,2}_*(X)$. Let $T$ be a closed positive $(1,1)$-current on $X$ such that $du\wedge d^c u \leq T$. Then there exists a sequence $(u_k)$ of smooth functions on $X$ such that $u_k \rightarrow u$ pointwise outside a pluripolar set. Moreover, there exist closed positive $(1,1)$-forms $T_k$ and a constant $C$ depending only on $X,\omega$ such that
    \[d u_k\wedge d^c u_k \leq T_k \qquad \text{and} \qquad \|T_k\|_X \leq C\|T\|_X.\]
    If $u$ is bounded, then $u_k$ are uniformly bounded.
\end{theorem}

The following corollary is useful for us.

\begin{corollary}\label{convergence}
    Let $\mu$ be a non-pluripolar finite Borel measure on $X$. Let $u$ be a bounded function in $W^{1,2}_*(X)$, and let $(u_k)$ be the sequence in Theorem~\ref{regularize}. Then
    \[\lim_{k\rightarrow \infty} \int_X |u_k - u | d\mu = 0.\]
\end{corollary}

\begin{proof}
    Since $\mu$ is non-pluripolar, by Theorem~\ref{regularize}, we have $u_k$ converges to $ u$ $\mu$-almost everywhere. Since $|u|,|u_k|$ are uniformly bounded, the result follows from the dominated convergence theorem. 
\end{proof}

To prove Theorem~\ref{regularize}, we first recall the regularization method of Dinh-Sibony in \cite{DS_regula}. Let $\Delta$ denote the diagonal of $X\times X$. Let $ \pi: \widetilde{X\times X} \to X\times X$ be the blow up of $X\times X$ along $\Delta$. Let $\widetilde{\Delta}:=\pi^{-1}(\Delta)$ be the exceptional divisor. Let $\gamma$ be a closed, strictly positive $(n-1,n-1)$-form on $\widetilde{X\times X}$. We note that $\pi_*(\gamma \wedge [\widetilde{\Delta}])$ is a closed positive current of bi-dimension $(n,n)$ supported on $\Delta$. Hence $\pi_*(\gamma \wedge [\widetilde{\Delta}]) = c [\Delta]$ for some constant $c > 0$. By rescaling, we can assume that $\pi_*(\gamma \wedge [\widetilde{\Delta}]) = [\Delta]$.

Since $[\widetilde{\Delta}]$ is a closed positive $(1,1)$-current, we can pick a quasi-psh function $\varphi$ and a smooth closed $(1,1)$-form $\Theta'$ such that $\sup_{\widetilde{X\times X}} \varphi = 0$ and $dd^c\varphi=[\widetilde{\Delta}]-\Theta'$. Note that $\varphi$ is smooth out of $\widetilde{\Delta}$ and $\varphi^{-1}(-\infty)=\widetilde{\Delta}$. Let $\chi:\R\rightarrow\R$ be a smooth, increasing, convex function such that $\chi\equiv0$ on $(-\infty,-1]$, $\chi(t)=t$ on $[1,+\infty)$, and $0\leq\chi'(t)\leq 1$. Define $\chi_k(t):=\chi(t+k)-k$ and $\varphi_k=\chi_k\circ \varphi$. We have $\varphi_k$ is smooth outside $\widetilde{\Delta}$ and equal $-k$ on the open neighborhood $\{\varphi < - k - 1\}$ of $\widetilde{\Delta}$. Thus $\varphi_k$ is smooth for every $k$. Since $0\leq \chi'(t) \leq 1$, we get $\chi(t) \geq \chi(t+1)-1$. Therefore $\varphi_k \geq \varphi_{k+1}$. Since $\sup_{\widetilde{X\times X}} \varphi = 0$, we infer that $\varphi_k$ decreases to $\varphi$.

Let $\Theta$ be a smooth closed positive $(1,1)$-form such that $\Theta - \Theta'$ is positive. We have
\[dd^c\varphi_k=(\chi_k''\circ\varphi)d\varphi\wedge d^c\varphi+(\chi_k'\circ\varphi)dd^c\varphi\geq (\chi_k'\circ\varphi)dd^c\varphi=-(\chi_k'\circ\varphi)\Theta'\geq-\Theta.\]
Define $\Theta_k^+:=dd^c\varphi_k+\Theta$ and $\Theta^-_k:=\Theta-\Theta'$. Then $\Theta_k^+,\Theta_k^-$ are closed positive $(1,1)$-forms and $\Theta_k^+-\Theta_k^-$ converges weakly to $[\widetilde{\Delta}]$ as currents. Define
\[\widetilde{K_k^\pm} :=\gamma\wedge \Theta_k^\pm \quad \text{ and } \quad K_k^\pm:=\pi_*\Big(\widetilde{K_k^\pm}\Big).\]
Then $K_k^\pm$ are closed positive $(n,n)$-currents and smooth outside $\Delta$ (thus are $L^1$ forms). Since $\pi_*(\gamma \wedge [\widetilde{\Delta}]) = [\Delta]$, we get that $K_k^+-K_k^- $ converges to $[\Delta]$ weakly as currents. Now, we put $\widetilde{K_k}:=\widetilde{K_k^+}-\widetilde{K_k^-}$ and $K_k:=K_k^+-K_k^-$.

The kernel $K_k$ was used by Vigny in the proof of Theorem~\ref{regularizeVig} as well as in the proof of \cite[Theorem 1.1]{DS_regula}. We now prove some properties of this kernel.

\begin{proposition}\label{suppKblow}
    $\supp \Big(\widetilde{K_k}\Big) \subset \varphi^{-1}([-\infty,-k+1])$ and thus converges to $\widetilde{\Delta}$ as $k\rightarrow \infty$.
\end{proposition}

\begin{proof}
Recall that $\widetilde{K_k} = \gamma \wedge (\Theta_k ^+ -\Theta_k ^-)$ where
\[\Theta_k ^+ -\Theta_k ^- = \Theta' + dd^c\varphi_k=\Theta'+ (\chi_k''\circ\varphi)d\varphi\wedge d^c\varphi+(\chi_k'\circ\varphi)dd^c\varphi.\]
By definition, we have $\chi''_k\equiv0$ outside $(-k-1,-k+1)$ and $\chi_k'\equiv 1$ on $[-k+1,\infty)$. This deduces that $\Theta' + dd^c \varphi_k = \Theta' + dd^c \varphi $ on $\varphi^{-1}((-k+1,\infty))$ which is an open subset since $\varphi$ is smooth outside $\widetilde{\Delta}$. Since $\Theta' +  dd^c\varphi=[\widetilde{\Delta}]$, $\Theta' + dd^c\varphi_k $ must have support inside $\varphi^{-1}([-\infty,-k+1])$. Therefore, $\widetilde{K_k}$ must have support inside $\varphi^{-1}([-\infty,-k+1])$.
\end{proof}

Since we want pointwise convergence, we only need to work on a local chart around a point of $X$. Fix a point $x_0 \in X$ and consider a chart of $X\times X$ centered at $(x_0,x_0)$. Let $(x,y)=(x_1,\ldots,x_n,y_1,\ldots,y_n)$, $|x_j|<3$, $|y_j|<3$, for $j=1,\ldots,n$, be local holomorphic coordinates of this chart  such that $\Delta=\{y=0\}$.  This chart can be covered by $n$ sectors $S_1,\ldots,S_n$, where $S_l$ is defined by $|x_j|<3$, $|y_j|<3|y_l|$ for $j\in\{1,\ldots,n\}\setminus\{l\}$, and $l=1,\ldots,n$.

The blow-up of this chart can be covered by $n$ charts $\widetilde{S_1},\ldots,\widetilde{S_n}$ where $\widetilde{S_l}$ is defined by $|x_j|<3$, $\widetilde{y_l}=y_l$, $\widetilde{y_j}=y_j/y_l$, $|y_l|<3$, $|y_j|<3|y_l|$ for $j\in\{1,\ldots,n\}\setminus\{l\}$ and $l=1,\ldots,n$. The equation for $\widetilde{\Delta}$ in $\widetilde{S_l}$ is $\widetilde{y_l} = 0$.

We want to control the support of $K_k$ on this chart.

\begin{proposition}\label{suppKk}
$\supp(K_k(0,\cdot))\subset B(0,Ce^{-k})$ for some uniform constant $C>0$.
\end{proposition}

\begin{proof}
By symmetry, we only need to consider $\widetilde{S_1}$. We can choose a local potential $v$ of $\Theta'$ such that in the coordinate $(x,\widetilde{y})$ of $\widetilde{S_1}$, $\varphi=\log|\widetilde{y_1}|-v$. By Proposition~\ref{suppKblow}, we have
\[\widetilde{S_1}\cap\supp\Big(\widetilde{K_k}\Big)\subset \{(x,\widetilde{y})\in\widetilde{S_1}:\ \log|\widetilde{y_1}|-v(x,\widetilde{y})\in  [-\infty,-k+1]\}.\] 
Since $v$ is smooth, we can bound it on $\widetilde{S_1}$ by some constant. This implies that $\widetilde{S_1}\cap\supp\Big(\widetilde{K_k}\Big)\subset\{(x,\widetilde{y})\in \widetilde{S_1}:\ |\widetilde{y_1}|<Ce^{-k}  \}$ for some uniform constant $C>0$. We push down to $X\times X$ to see that $S_1\cap\supp(K_k)\subset\{(x,y):\ |y_j|<e^{-k}\textnormal{ for }j=1,\ldots,n\}$ (note that on this sector, $|y_j|<3|y_1|$ for $j=2,\ldots,n$). Since $(x_0,x_0) = (0,0)$ in this coordinate, the result follows.
\end{proof}

Next, we want to control the singularity of $K_k$ on this chart. Let $H_k$ be a coefficient of $K_k$ in these coordinates. It is known by \cite[Lemma 3.1]{DS_regula} that 
\[H_k(0,y)\lesssim \dfrac{A_k}{|y|^{2n-2}} \]
for some constant $A_k$ depending only on $k$. Our next proposition quantifies explicitly $A_k$.

\begin{proposition}\label{singK}
    $H_k(0,y)\lesssim \dfrac{e^{2k}}{|y|^{2n-2}}\cdot$
\end{proposition}

\begin{proof} Since $\widetilde{K_k}$ is a smooth form on $\widetilde{S_1}$, we can write $\widetilde{K_k}$ as finite sums of forms of type
$$\Phi(x,\widetilde{y})=L(x,\widetilde{y})dx_I\wedge d\overline{x}_{I'}\wedge d\widetilde{y}_J\wedge d\overline{\widetilde{y}}_{J'},$$
where $L$ is a smooth function, $I,I',J,J'$ are subsets of $\{1,\ldots,n\}$ and 
$$dx_I=dx_{i_1}\wedge\cdots\wedge dx_{i_m}\text{ if }I=\{i_1,\ldots,i_m\}.$$

Let $v$ be the local potential of $\Theta'$ such that $\varphi=\log|\widetilde{y_1}|-v$ in the coordinate $(x,\widetilde{y})$ of $\widetilde{S_1}$. In this coordinate, we can write 
    \begin{align}\label{describe Kktilde}
        dd^c\varphi_k+\Theta' &= \Theta' + (\chi''_k\circ (\log|\widetilde{y_1}|-v))d(\log|\widetilde{y_1}|-v)\wedge d^c(\log|\widetilde{y_1}|-v)\\ \nonumber
    &+(\chi_k'\circ(\log|\widetilde{y_1}|-v))dd^c(\log|\widetilde{y_1}|)-(\chi_k'\circ(\log|\widetilde{y_1}|-v))dd^cv.
    \end{align}
     The first term $\Theta'$ is a fixed form. The second term only takes account when $\log|\widetilde{y_1}|-v\in[-k-1,-k+1]$. The third term contributes nothing since $dd^c(\log|\widetilde{y_1}|) = [\widetilde{y_1} = 0]$ and $\chi_k' \equiv 0$ in an open neighborhood of $\{\widetilde{y_1} = 0\}$. The last term is also uniformly bounded since $v$ is fixed and $|\chi_k'| \leq 1$.

    Recall that $\widetilde{K_k}=\gamma \wedge (dd^c\varphi_k+\Theta')$. We can use the above descriptions of \eqref{describe Kktilde} to control $L(x,\widetilde{y})$. Note that
    \[\partial \log |\widetilde{y_1}|=\frac{1}{2\widetilde{y_1}}d\widetilde{y_1},\
   \overline{\partial} \log |\widetilde{y_1}|=\frac{1}{2\overline{\widetilde{y_1}}}d\overline{\widetilde{y_1}}.\]
   We can then bound $L(x,\widetilde{y})$ by $1/|\widetilde{y_1}|^2$ (which takes into account when $\log|\widetilde{y_1}|-v\in[-k-1,-k+1]$). This deduces that $L(x,\widetilde{y})\lesssim e^{2k}$ on $\widetilde{S_1}$.

   We push down $\widetilde{K_k}$ to $S_1$. We see that $\pi_*(\Phi)$ is obtained from $\Phi$ by replacing $\widetilde{y_1}$ with $ y_1$, and $\widetilde{y_j}$ with $ y_j/y_1$ for $j = 2,\ldots,n$. There are at most $2k-2$ factors of the form $d(y_j/y_1) = dy_j/y_1 - y_jdy_1/y_1^2$ or their conjugate. Hence, the coefficients of $\pi_*(\Phi)$ on $S_1$ are finite sum of the type
$$L(x,y_1,y_2/y_1,\ldots,y_n/y_1)P(y)y_1^{-m}\overline{y}_1^{-p},$$
where $P$ is a homogeneous polynomial in $y$ and $\overline{y}$ (this polynomial only depends on $I,I',J,J'$) such that $\deg(P)+2n-2\geq m+p$. Recall also that on sector $S_1$, we have $|y_j| < 3|y_1|$ for $j = 2,\ldots,n$. Therefore, 
    \[H_k(0,y)\lesssim\frac{e^{2k}}{|y|^{2n-2}}\]
on $S_1$. By symmetry, the proof is complete.
\end{proof}

We also need the following lemma about bounded mean oscillation functions.

\begin{lemma}\label{Cor JohnNirenberg}\cite{JN61-bmo}
    Let $\Omega$ be a bounded domain in $\R^n$, and $u$ be a bounded mean oscillation function in $\Omega$. Let $p$ be a positive constant. Then there exists a positive constant $C$ depending only on $n,p$ such that
    \[\sup\limits_{Q}  \ \frac{1}{\Leb(Q)}\int_Q \Big|u-m_Q(u)\Big|^p\ d\Leb\leq C\|u\|^p_{\bmo(\Omega)}.\]
    Here, the supremum is taken over all the cubes $Q \subset \Omega$.
\end{lemma}

We now prove Theorem~\ref{regularize}.

\begin{proof}[Proof of Theorem~\ref{regularize}]
Let $T$ be the closed positive $(1,1)$-current such that $du\wedge d^c u \leq T$. Define $$u_k(x):=\int_{y\in X}u(y)K_k(x,y).$$
    We consider the following set
    \[A:= \{x\in X \text{ such that } x \text{ is a Lebesgue point of } u \text{ and } \nu(T,x) = 0\}.\]
    By \cite[Theorem 1.1]{Vigny-Vu-Lebesgue}, $X\setminus A$ is a pluripolar set. We claim that $u_k(x) \rightarrow u(x)$ for every $x \in A$ as $k\rightarrow \infty$. By Proposition~\ref{suppKk}, this is a local problem.

    Fix a point $x_0 \in A$ and consider a local holomorphic coordinate around $x_0$. Since $K_k(x,y) \to [\Delta]$ as $k\to \infty$, $K_k(0,y) $ will converges to the Dirac measure at $0$ as $k\to \infty$. Therefore, $\int_{y\in X} u(0) K_k(0,y) \to u(0)$ as $k\to \infty$. By Proposition~\ref{singK}, we only need to show that
    \[\lim\limits_{k\rightarrow \infty}\int_{B(0,e^{-k})}|u(x)-u(0)|\frac{e^{2k}}{|x|^{2n-2}}d\Leb(x)=0.\]
    Observe that
\begin{align*}
    \int_{B(0,e^{-k})}|u(x)-u(0)|\frac{e^{2k}}{|x|^{2n-2}}d\Leb(x)
    &\leq \int\limits_{B(0,e^{-k})}\Big|u(x)-m_{B(0,e^{-k})}(u)\Big|\frac{e^{2k}}{|x|^{2n-2}}d\Leb(x) \\&+
    \Big| m_{B(0,e^{-k})}(u)-u(0) \Big|\int_{B(0,e^{-k})}\frac{e^{2k}}{|x|^{2n-2}}d\Leb(x).
\end{align*}
The second term on the right-hand side converges to $0$ as $k\rightarrow\infty$ since $0$ is a Lebesgue point and $\displaystyle \int_{B(0,e^{-k})}\frac{e^{2k}}{|x|^{2n-2}}d\Leb(x)$ is bounded from above by a constant which only depends on $n$. To conclude, it is thus enough to show that the first term converges to $0$ as $k\rightarrow\infty$.

By H\"older's inequality, we have
\begin{align*}
    &\int_{B(0,e^{-k})}\Big|u(x)-m_{B(0,e^{-k})}(u)\Big|\frac{e^{2k}}{|x|^{2n-2}}d\Leb(x)\\ \nonumber
    &\leq e^{2k}\Big(\int_{B(0,e^{-k})}\Big|u(x)-m_{B(0,e^{-k})}(u)\Big|^{2n-1}d\Leb(x)\Big)^{\frac{1}{2n-1}}\Big(\int_{B(0,e^{-k})}\frac{1}{|x|^{2n-1}}d\Leb(x)\Big)^{\frac{2n-2}{2n-1}}\\ \nonumber
    &=C_1\Big(\frac{1}{\Leb(B(0,e^{-k}))}\int_{B(0,e^{-k})}\Big|u(x)-m_{B(0,e^{-k})}(u)\Big|^{2n-1}d\Leb(x)\Big)^{\frac{1}{2n-1}},
\end{align*}
where $C_1$ depends only on $n$.  By \cite[Proposition 6]{Vigny}, $u$ is a bounded mean oscillation function. Therefore, by Lemma~\ref{Cor JohnNirenberg} (we can take a change of variables at $0$), we have
\[ \int_{B(0,e^{-k})}\Big|u(x)-m_{B(0,e^{-k})}(u)\Big|\frac{e^{2k}}{|x|^{2n-2}}d\Leb(x)\leq C_2\|u\|_{\bmo(B(0,2e^{-k}))},\]
where $C_2$ depends only on $n$. 

For $r>0$, put $$\nu(T,x,r):= r^{2(1-n)} \int_{B(x,r)}T\wedge (dd^c |z|^2)^{n-1}.$$ 
It is known that $\nu(T,x,r)$ is an increasing function on $r$ and converges to the Lelong number $\nu(T,x)$ of $T$ at $x$ as $r\to 0^+$. By the definition of $\bmo$-norm and the proof of \cite[Proposition 6]{Vigny}, we have
$$\|u\|_{\bmo(B(0,2e^{-k}))}\leq C_3 \sup\limits_{B(y,s)\subset B(0,2e^{-k})}\sqrt{\nu(T,y,s)}\leq C_4 \sqrt{\nu\Big(T,0,4e^{-k}\Big)},$$
where $C_3,C_4$ depend only on $n$. This deduces the claim since we have $\nu(T,0)=0$.

By \cite[Theorem 10]{Vigny}, we can choose
\[T_k(x):= \|K_k(x,\cdot)\|_{L^1} \int_{y\in X}  (K_k^+(x,y) + K_k^-(x,y)) \wedge T(y).\]
Then $T_k$ is a closed positive $L^{1+\delta}$-form ($\delta$ can choose depending on $n$, see \cite[Lemma 2.1]{DS_regula}) such that $\|T_k\|_X \leq C_5 \|T\|_X$ ($C_5$ depends only on $X,\omega$) and $du_k \wedge d^c u_k \leq T_k$.

To get more regularity of $u_k$ and $T_k$, we need to apply this process several times. Thus, we need to show that pointwise convergence is preserved under this process. Define
\[A_k:= \{x\in X \text{ such that } x \text{ is a Lebesgue point of } u_k \text{ and } \nu(T_k,x) = 0\}.\]
We only need to show that $A\subset A_k$ for every $k$. This is clear as $K_k^+(x,\cdot),K_k^-(X,\cdot)$ are $L^1$-integrable functions.

Applying this process several times, we get a sequence of $\mathscr{C}^1$ functions $(u_k)$ and a sequence of closed positive $\mathscr{C}^1$ $(1,1)$-form $T_k$ such that $du_k\wedge d^c u_k \leq T_k$ (see \cite[Lemma 2.1]{DS_regula}). Let $\varepsilon_k$ be a sequence decreases to $0$. Then for each $k$ we can pick a smooth function $\widetilde{u_k}$ and a closed smooth $(1,1)$-form $\widetilde{T_k}$ such that 
\[\|\widetilde{u_k} - u_k\|_{\mathscr{C}^0} \leq \varepsilon_k, \quad d \widetilde{u_k}\wedge d^c \widetilde{u_k} \leq du_k \wedge d^c u_k + \varepsilon_k \omega,  \quad T_k -\varepsilon_k\omega \leq \widetilde{T_k} \leq  T_k +\varepsilon_k\omega .\]
We infer that $d \widetilde{u_k}\wedge d^c \widetilde{u_k} \leq du_k \wedge d^c u_k + \varepsilon_k \omega \leq \widetilde{T_k} + 2 \varepsilon_k\omega$. This is our desired sequence.

When $u$ is bounded, we can bound $\|u_k\|_{L^\infty}$ by $\|K_k\|_{L^1} \|u\|_{L^\infty}$. Note that by Proposition~\ref{singK} and since 
\[ \int_{B(0,e^{-k})}\frac{e^{2k}}{|x|^{2n-2}}d\Leb(x)\] 
is bounded from above by a constant which only depends on $n$, we can bound uniformly $\|K_k\|_{L_1}$. The proof is complete.
\end{proof}

\section{Integration by parts formula}\label{intbypart}

In this section, we will establish a new integration by parts formula for functions in the complex Sobolev space, generalization of results in \cite{Vu-diameter} to the big cohomology class setting. 
 Let $\theta$ be a closed smooth $(1,1)$-form in a big cohomology class.  We denote by $\PSH(X,\theta)$ the set of $\theta$-psh functions. Let $\theta_1,\ldots,\theta_m$ be closed smooth $(1,1)$-forms in big cohomology classes and $\varphi_j \in \PSH(X,\theta_j)$ for $j=1,\ldots, m$. We denote by
\[(\theta_1 + dd^c \varphi_1) \wedge \cdots \wedge (\theta_m + dd^c \varphi_m)\]
the \textit{non-pluripolar product} of closed positive $(1,1)$-currents $\theta_j + dd^c \varphi_j$ for $j = 1,\ldots,m$. In this paper, we always consider non-pluripolar products of currents.
%Define
%\[V_\theta := \sup \{u\in \PSH(X,\theta) \text{ such that } u\leq 0\}.\]
It is known that the sequence of positive currents
\[\mathbf{1}_{\bigcap\{\varphi_j > V_{\theta_j}-k\}} (\theta_1 + dd^c \max(\varphi_1,V_{\theta_1}-k)) \wedge \cdots \wedge (\theta_m + dd^c \max(\varphi_m,V_{\theta_m}-k)\]
is increasing to $(\theta_1 + dd^c \varphi_1) \wedge \cdots \wedge (\theta_n + dd^c \varphi_n)$. We refer the reader to \cite{BEGZ,Viet-generalized-nonpluri} for more properties of non-pluripolar products. Recall that a dsh function is a difference of two quasi-psh function (see \cite{DS_tm}).
Here is our main result of this section.

\begin{theorem}\label{integrationbypartdsh}
Let $(X,\omega)$ be a compact K\"ahler manifold of dimension $n$. Let $u,v_1,\ldots,v_m$ be bounded functions in $W^{1,2}_*(X)$. Let $\theta_1,\ldots,\theta_{n-1}$ be closed $(1,1)$-forms in big cohomology classes. Let $\varphi_j \in \PSH(X,\theta_j)$ for $j=1,\ldots,n-1$. Define
\begin{equation}\label{define R}R :=  (\theta_1 + dd^c \varphi_1 )\wedge \cdots \wedge ( \theta_{n-1} + dd^c \varphi_{n-1} ) .\end{equation}
Let $\psi$ be a bounded dsh function and $\rho$ be a smooth function on $\R^m$. Then we have
    \begin{align}\label{intbypart formula} \int_X  \rho(v_1,\ldots,v_m)du\wedge d^c \psi \wedge R &= 
-\sum_{j=1}^m \int_X u\partial_j \rho(v_1,\ldots,v_m)  d v_j \wedge d^c \psi \wedge R \\ \nonumber
&-\int_X u  \rho(v_1,\ldots,v_m)  dd^c \psi \wedge R.
\end{align}
\end{theorem}

We will define the product $d v\wedge \dc \psi \wedge R$  below. 
We now comment on the proof of Theorem~\ref{integrationbypartdsh}. In \cite{Vu-diameter}, the author proved \eqref{intbypart formula} in the case when $\psi$ is a \textit{bounded quasi-psh function} and $R$ is a product of closed positive $(1,1)$-currents with bounded potential. The idea was to prove the local case first (due to the good approximations in the local setting) and then pass to the global case using a standard gluing argument. However, since the non-pluripolar product is not always well-defined in the local case, we cannot use the same argument in our setting. Therefore, we work directly on the global case (in which the non-pluripolar product is always well-defined). 

We will prove Theorem~\ref{integrationbypartdsh} through several approximation steps. Note that \eqref{intbypart formula} is known when $u,v_1,\ldots,v_m$ are smooth functions and $\psi$ is a bounded quasi-psh function. For the \textit{first step}, we assume that $\psi$ is a bounded quasi-psh function. We then prove convergence lemmas using the sequence constructed in Theorem~\ref{regularize} for $u$ and $v_1,\ldots,v_m$ (see Lemma~\ref{generalconvergence}). A key note is that all the measures we consider are non-pluripolar which is why Theorem~\ref{regularize} is useful (see Corollary~\ref{convergence}). This allows us to prove \eqref{intbypart formula} in the case when $\psi$ is a bounded quasi-psh function (see Theorem~\ref{intbypastpsh}). For the \textit{second step}, we will approximate $\psi$ by dsh functions $\psi_k$ which are difference of bounded quasi-psh functions (see Lemma~\ref{key}).

We now go into details. Let us begin with some definitions. Let $u$ be a \textit{bounded} function in $W^{1,2}_*(X)$. Let $\theta_1,\ldots,\theta_{n-1}$ be closed $(1,1)$-forms in big cohomology classes. Let $\varphi_j \in \PSH(X,\theta_j)$ for $j=1,\ldots,n-1$. Let $R = (\theta_1 + dd^c \varphi_1 )\wedge \cdots \wedge ( \theta_{n-1} + dd^c \varphi_{n-1} )$. For every smooth $1$-form $\Phi$, we put
\[\left \langle du\wedge R,\Phi\right \rangle : = - \int_X u d\Phi \wedge R.\]
Since $d\Phi \wedge R$ is a non-pluripolar measure on $X$ and a good representative is well-defined modulo a pluripolar set, the above integral makes sense. Moreover, as $u$ is a bounded function, $du\wedge R$ is a well-defined current on $X$. We define $\partial u \wedge R$ to be the $(n,n-1)$-current such that $\left \langle \partial u \wedge R,\Phi \right \rangle = \left \langle d u \wedge R,\Phi \right \rangle$ for every smooth $(0,1)$-form $\Phi$, $\overline{\partial} u \wedge R$ to be the $(n-1,n)$-current such that $\left \langle \overline{\partial} u \wedge R,\Phi \right \rangle = \left \langle d u \wedge R,\Phi \right \rangle$ for every smooth $(1,0)$-form $\Phi$. We have $du\wedge R = \partial u \wedge R+\overline{\partial} u\wedge R$. 

Consider a finite cover of $X$ by local charts $\{\mathcal{U}_\alpha\}_\alpha$, and fix a partition of unity $\{\chi_\alpha\}_\alpha$ for $\{\mathcal{U}_\alpha\}_\alpha$. For a form $\Phi$ on $X$ and $\beta \geq 0$, we denote by $\|\Phi\|_{\mathscr{C}^\beta}$ the sum of $\mathscr{C}^\beta$-norms of the coefficients of $\Phi$ on this fixed atlas.

\begin{lemma}\label{duwedgeR} Let $(u_k)_k$ be the sequence in Theorem~\ref{regularize} associated to $u$. Then $du_k \wedge R \rightarrow du\wedge R$ weakly as currents. We have similar convergences for $\partial u$ and $\overline{\partial} u$. In particular, $du\wedge R$, $\partial u \wedge R$, $\overline{\partial} u \wedge R$ are currents of order $0$.
\end{lemma}

\begin{proof} We argue similarly to \cite[Lemma 2.8]{Vu-diameter}. Let $\Phi$ be a $1$-form on $X$. By definition, we have
\[|\left \langle d u_k \wedge R - d u\wedge R,\Phi \right \rangle| = \Big| \int_{X} (u_k - u)R\wedge d\Phi \Big|\lesssim \|\Phi\|_{\mathscr{C}^1}\int_{X} |u_k - u| R\wedge \omega.\]
Here $\lesssim$ depends only on $X,\omega$ and $\{\mathcal{U}_\alpha\}_\alpha$. Since $R $ is non-pluripolar, $R\wedge \omega$ is a non-pluripolar measure. Thus, by Corollary~\ref{convergence},
\[|\left \langle d u_k \wedge R - d u\wedge R,\Phi \right \rangle| \rightarrow 0 \text{ as }k\rightarrow \infty.\]
By considering $\Phi$ as $(0,1)$ and $(1,0)$ form, we get the same conclusion for $\partial u \wedge R$ and $\overline{\partial} u \wedge R$. The first part follows.

Now, we prove the second part. We have $i\partial u_k \wedge \overline{\partial} u_k \leq T_k$ on $X$ where $T_k$ is a closed positive $(1,1)$-form and $\|T_k\|_X \leq C_1 \|T\|_X$ for some constant $C_1$ that does not depend on $u$ and $k$. Let $v$ be a smooth function on $X$ such that $i\partial v\wedge \overline{\partial} v \leq \omega$. By Cauchy-Schwarz inequality, we have
\begin{align*}\Big|\int_{X} i\partial u_k \wedge \overline{\partial} v \wedge R \Big| &\leq \Big(\int_{X} i\partial u_k \wedge \overline{\partial} u_k \wedge R\Big)^{1/2} \cdot \Big(\int_{X} i\partial v \wedge \overline{\partial} v \wedge R\Big)^{1/2} \\
&\leq \Big(\int_{X} T_k \wedge R\Big)^{1/2} \|R\|_X^{1/2} 
\lesssim \|T_k\|_X^{1/2} \|R\|_X^{1/2}
\leq C_1^{1/2} \|T\|_X^{1/2} \|R\|_X^{1/2}.\end{align*}
Here, $\lesssim$ depends only on $\theta_1,\ldots,\theta_{n-1}$. We note that at this point, we bound these forms by some constant times $\omega$ and use the monotonicity of the non-pluripolar product (see, e.g., \cite[Theorem 1.1]{Lu-Darvas-DiNezza-mono}) to obtain:
\[\int_X T_k \wedge R \leq \int_X T_k \wedge (C\omega + dd^c \varphi_1) \wedge \cdots \wedge (C\omega+dd^c\varphi_{n-1}) \leq \int_X T_k \wedge (C\omega)^{n-1} = C^{n-1} \|T_k\|_X.\]
Letting $k\rightarrow \infty$, we get the bound
\begin{equation}\label{bound partial u wedge partial bar v R}\Big|\int_{X} i\partial u \wedge \overline{\partial} v \wedge R \Big| \leq C_2,\end{equation}
where $C_2$ is a constant that does not depend on $v$.

Let $\Phi$ be a $(0,1)$-form on $X$. Write $ \Phi|_{\mathcal{U}_\alpha} = \sum_{j=1}^n f_{\alpha,j} \overline{\partial}z_j $ where $|f_{\alpha,j}| \leq \|\Phi|_{\mathcal{U}_\alpha}\|_{\mathscr{C}^0} \leq \|\Phi\|_{\mathscr{C}^0}$. Put $v_j = \sum_\alpha \chi_\alpha z_j$. Then $v_j$ is a well-defined smooth function on $X$. We have
\[\Phi = \sum_\alpha \sum_{j=1}^n\chi_\alpha f_{\alpha,j} \overline{\partial} z_j  =\sum_\alpha \sum_{j=1}^n f_{\alpha,j} (\overline{\partial} v_j - z_j \overline{\partial} \chi_\alpha).\]
Since $\partial v_j = z_j \partial \chi_\alpha + \chi_\alpha \partial z_j$ and $\overline{\partial} v_j = z_j \overline{\partial} \chi_\alpha + \chi_\alpha \overline{\partial} z_j$ on $\mathcal{U_\alpha}$, we can bound $i\partial v_j \wedge \overline{\partial} v_j \leq C_3 \omega$ by some constant $C_3$ depending only on $\{\mathcal{U}_\alpha\}_\alpha$. Moreover, we can bound $|z_j|$ by some constant $C_4$ and $i\partial \chi_\alpha \wedge \overline{\partial} \chi_\alpha \leq C_5 \omega$ for some constant $C_5$. Both $C_4$ and $C_5$ depend only on $\{\mathcal{U}_\alpha\}_\alpha$. Using \eqref{bound partial u wedge partial bar v R}, we get
\[
\Big|\int_{X} i\partial u \wedge \Phi \wedge R \Big| \lesssim \|\Phi\|_{\mathscr{C}^0}
\]
where $\lesssim$ does not depend on $\Phi$. Thus, $\partial u \wedge R$ is a current of order $0$. Similarly, $\overline{\partial} u \wedge R$ is also a current of order $0$. The proof is complete.
\end{proof}

\subsection{First step} Let $\psi$ be a \textit{bounded} non-negative quasi-psh function on $X$. Let $C_\psi>0$ such that $\psi$ is a $C_\psi\omega$-psh function.  We define $du \wedge d^c \psi \wedge R = d\psi \wedge d^c u \wedge R$ by
\[\left \langle du \wedge d^c \psi \wedge R,\chi \right \rangle := -\int_X u d\chi \wedge d^c \psi \wedge R - \int_X u \chi dd^c \psi \wedge R\]
for smooth function $\chi$ on $X$. When $u$ is smooth, this notion coincides with the usual wedge product by classical integration by parts formula.

This is a well-defined current since $u$ is a bounded function and $\psi$ is a bounded quasi-psh function (which implies that both $d\chi \wedge d^c \psi \wedge R$ and $\chi dd^c \psi \wedge R$ are non-pluripolar measures).

\begin{lemma}\label{dudcwedgeR}
    Let $(u_k)_k$ be the sequence in Lemma~\ref{regularize} associated to $u$. Then $du_k \wedge d^c \psi \wedge R \rightarrow du \wedge d^c \psi\wedge R$ weakly as currents. In particular, $du\wedge d^c \psi\wedge R$ is a current of order $0$ and hence a signed measure.
\end{lemma}

\begin{proof}
    Let $\chi$ be a smooth function on $X$. We have
    \begin{align*}|\left \langle d(u_k-u)\wedge d^c \psi \wedge R ,\chi\right \rangle| 
&= \Big| \int_X (u_k-u) d\chi\wedge d^c \psi \wedge R + \int_X (u_k-u)dd^c \psi \wedge R\Big| \\
&\leq  \int_X |u_k-u| |d\chi\wedge d^c \psi \wedge R| + \int_X |u_k-u| |dd^c \psi\wedge R|.
\end{align*}
Here, for a signed measure $\mu$, we denote $|\mu|$ as the total variation of $\mu$. Since $\psi$ is a bounded quasi-psh function, both $|d\chi\wedge d^c \psi \wedge R|$ and $|dd^c \psi \wedge R|$ put no mass on pluripolar sets. Hence, by Corollary~\ref{convergence}, the first part follows.

We now prove the second part. By Cauchy-Schwarz inequality, we have
\[\Big| \int_X \chi du_k\wedge d^c \psi \wedge R \Big| \leq \Big|\int_X du_k \wedge d^c u_k \wedge R \Big|^{1/2} \Big|\int_X \chi^2 d\psi \wedge d^c \psi \wedge R \Big|^{1/2}\lesssim \|\chi\|_{\mathscr{C}^0}.\]
Here, $\lesssim$ does not depend on $\chi$ (we argue similarly to Lemma~\ref{duwedgeR}). Letting $k\rightarrow \infty$, we get
\[\Big| \int_X \chi du\wedge d^c \psi \wedge R \Big| \lesssim \|\chi\|_{\mathscr{C}^0}.\]
Hence, $du\wedge d^c \psi \wedge R$ is a current of order $0$. The proof is complete.
\end{proof}

\begin{corollary}\label{boundbycap}
   There exists a continuous function $f:[0,\infty)\rightarrow [0,\infty)$ satisfying $f(0) = 0$, depending on $X,\omega,R,\|u\|_*,C_\psi,\|\psi\|_{L^
   \infty}$, such that
    \begin{equation}\label{bound by cap formula}|du_k\wedge d^c \psi \wedge R|(E) \leq f ( \capa_\omega (E) )\end{equation}
    for every $k$ and for every Borel subset $E$ of $X$. As a consequence, the total variation $\mu$ of $du \wedge d^c \psi \wedge R$ is a non-pluripolar measure.
\end{corollary}

\begin{proof}
Let $E$ be a Borel subset of $X$. By Cauchy-Schwarz inequality, we have
    \begin{align*} \int_E  |du_k\wedge d^c \psi \wedge R | &\leq \Big|\int_E du_k \wedge d^c u_k \wedge R \Big|^{1/2} \Big|\int_E d\psi \wedge d^c \psi \wedge R \Big|^{1/2} \\
    &\leq \Big|\int_E T_k \wedge R \Big|^{1/2} \Big|\int_E d\psi \wedge d^c \psi \wedge R \Big|^{1/2} 
    \lesssim \Big|\int_E d\psi \wedge d^c \psi \wedge R \Big|^{1/2},\end{align*}
  where $\lesssim$ depends only on $X,\omega,R,\|u\|_*$. We note that here we need to use the monotonicity of the non-pluripolar product to bound $\int_X T_k \wedge R$ by some uniform constant (see also the proof of Lemma~\ref{duwedgeR}). Now, since $\psi$ is a bounded quasi-psh function and $d\psi \wedge d^c \psi = dd^c (\psi^2) - \psi dd^c \psi$, we can bound
  \begin{align*}d\psi \wedge d^c \psi \wedge R &\leq ((C_\psi +\|\psi\|_{L^\infty})\omega + dd^c (\psi + \psi^2)) \wedge R \\ &\leq \frac{1}{n}\cdot \Big((C_\psi +\|\psi\|_{L^\infty})\omega + \sum_{j=1}^{n-1}\theta_j + dd^c (\psi + \psi^2 + \sum_{j=1}^{n-1}\varphi_j)\Big)^n \\
  &\leq \frac{1}{n} \cdot (\|\psi\|_{L^\infty}+\|\psi^2\|_{L^\infty})\capa_{\Theta,\Psi}(\cdot).\end{align*}
  Here $\capa_{\Theta,\Psi}$ is the Monge-Amp\`ere capacity introduced in \cite{DiNezzaLu-capacity}, defined by:
  \begin{align*}
\capa_{\Theta,\Psi}(E):= &\sup \Big\{\int_E (\Theta + dd^c u)^n: u\in \PSH(X,\Theta): \Psi-1\leq u\leq \Psi\Big\}, \\&\text{ with }
  \Theta = (\|\psi\|_{L^\infty}+\|\psi^2\|_{L^\infty})^{-1} ((C_\psi +\|\psi\|_{L^\infty})\omega + \sum_{j=1}^{n-1}\theta_j),\\ &\text{ and }
  \Psi = (\|\psi\|_{L^\infty}+\|\psi^2\|_{L^\infty})^{-1} \sum_{j=1}^{n-1}\varphi_j.
\end{align*}
  Then, by \cite[Theorem 1.1]{Lu-comparison-capacity}, we get that the function $f$ satisfies \eqref{bound by cap formula}.

  Assume that $E$ is a pluripolar set. Let $E_l$ be a decreasing sequence of open sets such that $E = \bigcap_{l} E_l $. By Lemma~\ref{dudcwedgeR}, we have
    \[\mu(E_l) \leq \liminf_{k\to \infty} \int_{E_l} |du_k\wedge d^c \psi \wedge R|.\]
    By \eqref{bound by cap formula}, we have $\mu(E_l) \leq f(\capa_\omega(E_l))$ for every $l$. Letting $l \rightarrow \infty$, we infer that $\mu$ is a non-pluripolar measure. The proof is complete.
\end{proof}

\begin{lemma}\label{generalconvergence}
    Let $v_1,\ldots,v_m$ be bounded functions in $W^{1,2}_*(X)$. Let $v_{j,k}$ be the sequence as in Lemma \ref{regularize} associated to $v_j$ for $j=1,\ldots,m$. Let $\rho$ be a smooth function on $\R^m$. Then
    \[\rho(v_{1,k},\ldots,v_{m,k})du_k \wedge d^c \psi \wedge R \rightarrow \rho(v_1,\ldots,v_m) du \wedge d^c \psi \wedge R\]
    weakly as currents.
\end{lemma}

\begin{proof}
Recall that a good representative is well-defined modulo a pluripolar set. Then, by Corollary~\ref{boundbycap}, all the currents make sense. Let $\chi$ be a smooth function on $X$. By quasi-continuity property of complex Sobolev functions (see \cite{Vigny}), for $\varepsilon>0$, there exists an open set $E$ such that $\capa_\omega ( E) < \varepsilon$, and a continuous function $\widetilde{v_1},\ldots,\widetilde{v_m}$ such that $v_j = \widetilde{v_j}$ on $X\setminus E$ for $j=1,\ldots,m$. Therefore, by Lemma~\ref{dudcwedgeR}, we have
    \[\int_X \chi \rho(\widetilde{v_{1}},\ldots,\widetilde{v_{m}})du_k \wedge d^c \psi \wedge R \rightarrow \int_X \chi \rho(\widetilde{v_{1}},\ldots,\widetilde{v_{m}}) du \wedge d^c \psi \wedge R.\]

    On the other hand, by Corollary~\ref{boundbycap}, we can control
    \begin{align*}
\Big|\int_X \chi \rho(\widetilde{v_{1}},\ldots,\widetilde{v_{m}})&du_k \wedge d^c \psi \wedge R - \int_X \chi \rho(v_1,\ldots,v_m) du_k \wedge d^c \psi \wedge R \Big| \\
&= \Big|\int_E \chi (\rho(\widetilde{v_{1}},\ldots,\widetilde{v_{m}})-\rho(v_1,\ldots,v_m) ) du_k \wedge d^c \psi \wedge R \Big|\\
&\leq 2 \|\chi\|_{L^\infty}\|\rho\|_{L^\infty([-M_1,M_1]\times \cdots\times[-M_m,M_m])} f (\capa_\omega(E)),
 \end{align*}
 where $M_j = \|v_j\|_{L^\infty}$ for $j=1,\ldots,m$. Since $E$ is open, we also have $|du\wedge d^c \psi \wedge R|(E) \leq f (\capa_\omega(E))$. Noting that $f$ is continuous and $f(0) = 0$, letting $\varepsilon \rightarrow 0^+$, we get
 \[\rho(v_{1},\ldots,v_{m})du_k \wedge d^c \psi \wedge R \rightarrow \rho(v_1,\ldots,v_m) du \wedge d^c \psi \wedge R\]
 weakly as currents.

 The remaining is to control
 \[\Big|\int_X \chi \rho(v_{1,k},\ldots,v_{m,k})du_k \wedge d^c \psi \wedge R - \int_X \chi \rho(v_1,\ldots,v_m) du_k \wedge d^c \psi \wedge R \Big|.\]
 By Corollary~\ref{boundbycap}, we can bound $|du_k\wedge d^c \psi\wedge R|$ uniformly by $f(\capa_\omega (\cdot) )$. Now, arguing similarly to Lemma~\ref{convergence}, the proof is complete.
\end{proof}

We finish the first step by proving the following integration by parts formula.

\begin{theorem}\label{intbypastpsh}
    Let $u,v_1,\ldots,v_m$ be bounded functions in $W^{1,2}_*(X)$. Let $\theta_1,\ldots,\theta_{n-1}$ be closed $(1,1)$-forms in big cohomology classes. Let $\varphi_j \in \PSH(X,\theta_j)$ for $j=1,\ldots,n-1$. Let $R= (\theta_1+dd^c \varphi_1) \wedge \cdots \wedge (\theta_n+dd^c \varphi_n)$. Let $\psi$ be a bounded quasi-psh function. Let $\rho$ be a smooth function in $\R^m$. Then we have
    \begin{align*} \int_X  \rho(v_1,\ldots,v_m) du\wedge d^c \psi \wedge R &= 
-\sum_{j=1}^m \int_X u \partial_j\rho(v_1,\ldots,v_m) d v_j \wedge d^c \psi \wedge R\\ &- \int_X u \rho(v_1,\ldots,v_m) dd^c \psi \wedge R.
\end{align*}
\end{theorem}

\begin{proof}
    We use Lemma~\ref{regularize} and Lemma~\ref{generalconvergence} to pass to the case when all $u,v_1,\ldots,v_m$ are smooth functions. The result follows by the usual integration by part formula.
\end{proof}

\begin{remark}\label{intbypartdshbound}
    We note that in Theorem~\ref{intbypastpsh}, one can consider $\psi = \gamma_1 - \gamma_2$, where $\gamma_1$ and $\gamma_2$ are bounded quasi-psh functions.
\end{remark}

\subsection{Second step}\label{bounded dsh section}

Consider the case when $\psi$ is merely a bounded dsh function. Put $\psi = \gamma_1 - \gamma_2$, where $\gamma_1$ and $\gamma_2$ are quasi-psh functions. We define $dd^c \psi \wedge R = dd^c \gamma_1 \wedge R - dd^c \gamma_2 \wedge R$ in the non-pluripolar sense. Let $\psi_k = \max(\gamma_1,-k) - \max(\gamma_2,-k)$. By \cite[Lemma 2.6]{Viet-convexity-weightedclass} (see also \cite[Proposition 2.4]{Vu_DoHS-quantitative1}), the sequence of positive measures 
\[\mathbf{1}_{\{\gamma_1 >-k\} \cap \{\gamma_2>-k\}} d\psi_k \wedge d^c \psi_k \wedge R\]
has uniformly bounded mass and converges to a non-pluripolar measure which we denote by $d\psi \wedge d^c \psi \wedge R$. The key of this fact is the following estimate (which we also use frequently later):

\begin{lemma}\label{dv wedge dc v control}
    Let $u_1$ and $u_2$ be bounded $C\omega$-psh functions. Let $T$ be a closed positive $(n-1,n-1)$-current. Put $v = u_1 - u_2$. Then we have
    \[\int_X dv\wedge d^c v \wedge T \leq 2 C \|v\|_{L^\infty} \|T\|_X.\]
\end{lemma}

\begin{proof}
    By Stokes' theorem, we have 
    \begin{align*}\int_X dv\wedge d^c v \wedge T  & = -\int_X v dd^c v \wedge T \\
&=-\int_X v (dd^c u_1 +C\omega) \wedge T + \int_X v (dd^c u_2 +C\omega) \wedge T  \\
 &\leq 2 C \|v\|_{L^\infty} \|T\|_X, \end{align*}
 as desired.
\end{proof}

Let $\chi$ be a smooth function on $X$, we can define 
\[2d\chi \wedge d^c \psi \wedge R:= d(\chi+\psi)\wedge d^c (\chi+\psi)\wedge R - d\chi\wedge d^c\chi \wedge R - d\psi\wedge d^c\psi \wedge R.\]
Put $\phi_k = k^{-1} \max (\gamma_1+\gamma_2, -k)+1$. The following convergence result is useful for us (see \cite[proof of Theorem 2.7]{Viet-convexity-weightedclass}):
\begin{equation} \label{convergencenonpluri1}
d \chi \wedge d^c \psi \wedge R = \lim_{k\rightarrow \infty} \phi_k d\chi \wedge d^c \psi_k \wedge R \quad \text{and} \quad \chi  dd^c \psi \wedge R  = \lim_{k\rightarrow \infty} \chi \phi_k dd^c \psi_k \wedge R.
\end{equation}

Let $u$ be a bounded function in $W^{1,2}_*(X)$. Define
\[\left \langle d u \wedge d^c \psi \wedge R ,\chi\right \rangle := -\int_X u d \chi \wedge d^c \psi \wedge R  - \int_X u \chi   dd^c \psi \wedge R.\]
When $u$ is a bounded quasi-psh function (which is still an element in $W^{1,2}_*(X)$), this notion coincides with the notion introduced in \cite{Viet-convexity-weightedclass} by \cite[Theorem 2.7]{Viet-convexity-weightedclass}. This is a well-defined current since $u$ is a bounded function and both $d \chi \wedge d^c \psi \wedge R$ and $\chi   dd^c \psi \wedge R$ are non-pluripolar measures.

\begin{lemma}\label{convergencenonpluri}
    $ d u \wedge d^c \psi \wedge R =\displaystyle \lim_{k\rightarrow \infty} \phi_k du\wedge d^c\psi_k \wedge R$ 
    weakly as currents.
\end{lemma}

\begin{proof}
Let $\chi$ be a smooth function on $X$. We only need to show
\begin{equation}\label{convergencenonpluri 1}\int_X \chi  d u \wedge d^c \psi \wedge R = \lim_{k\rightarrow \infty} \int_X \chi \phi_k du\wedge d^c\psi_k \wedge R.\end{equation}
By Cauchy-Schwarz, we have
\begin{align*}d (\chi \phi_k )\wedge d^c (\chi \phi_k) &= \phi_k^2 d\chi \wedge d^c \chi + \chi^2 d\phi_k \wedge d^c \phi_k + \phi_k \chi (d\chi \wedge d^c\phi_k + d\phi_k\wedge d^c \chi) \\
&\leq \phi_k^2 d\chi \wedge d^c \chi + \chi^2 d\phi_k \wedge d^c \phi_k + \phi_k \chi (d\chi \wedge d^c\chi + d\phi_k\wedge d^c \phi_k).
\end{align*}
Since $ 0\leq \phi_k\leq 1$ and $\phi_k$ is a bounded quasi-psh function, we infer that $\chi \phi_k$ is a bounded function in $W^{1,2}_*(X)$. Now, by Theorem~\ref{intbypastpsh}, we get
\begin{align}\label{convergencenonpluri 2} \int_X \chi \phi_k du\wedge d^c\psi_k \wedge R = &-\int_X u \chi d\phi_k \wedge d^c \psi_k \wedge R - \int_X u \phi_k d\chi \wedge d^c \psi_k \wedge R  \\ \nonumber &- \int_X u \chi \phi_k dd^c \psi_k \wedge R.
\end{align}

Let $I_{1,k},I_{2,k}$ and $I_{3,k}$ denote the first, second, and third term of the right-hand side of \eqref{convergencenonpluri 2}. We first estimate $I_{1,k}$. Let $C$ be a positive constant such that $\gamma_1$ and $\gamma_2$ are $C\omega$-psh functions. By Cauchy-Schwarz and Lemma~\ref{dv wedge dc v control}, we have
\begin{align*}\Big| \int_X u \chi d\phi_k \wedge d^c \psi_k \wedge R \Big| &\leq \|u\chi\|_{L^\infty} \Big|\int_X d\phi_k\wedge d^c \phi_k \wedge R \Big|^{1/2} \Big|\int_X d\psi_k\wedge d^c \psi_k \wedge R \Big|^{1/2} \\
&\leq \Big|\int_X d\phi_k \wedge d^c \phi_k \wedge R \Big|^{1/2}\cdot (2 C \|\psi\|_{L^\infty} \|R\|_X)^{1/2}.
\end{align*}
Write $d\phi_k \wedge d^c \phi_k = dd^c (\phi_k)^2 - \phi_k dd^c \phi_k$. Since $\phi_k$ increases to the constant function $1$ on $X$ outside the set $\{\gamma_1 = -\infty\} \cup \{\gamma_2 = -\infty\}$, by \cite[Theorem 2.2]{Viet-convexity-weightedclass}, $I_{1,k}\rightarrow 0$ as $k\rightarrow \infty$.

Next, we estimate $I_{2,k}$. Let $(u_l)_l$ be the sequence in Lemma~\ref{regularize} associated to $u$. By \eqref{convergencenonpluri1}, we have
\[\int_X u_l \phi_k d\chi \wedge d^c \psi_k \wedge R \rightarrow \int_X u_l d\chi \wedge d^c \psi \wedge R\]
as $k \rightarrow \infty$. By Cauchy-Schwarz and Lemma~\ref{dv wedge dc v control}, we have
\begin{align*} \Big|\int_X (u_l-u) \phi_k d\chi \wedge d^c \psi_k \wedge R \Big| &\leq \Big| \int_X |u_l-u|^2 d\chi\wedge d^c \chi \wedge R \Big|^{1/2} \Big|\int_X d\psi_k \wedge d^c \psi_k \wedge R \Big|^{1/2}\\
&\leq \Big| \int_X |u_l-u|^2 d\chi\wedge d^c \chi \wedge R \Big|^{1/2} \cdot (2 C \|\psi\|_{L^\infty} \|R\|_X)^{1/2}.
\end{align*}
Since $d\chi \wedge d^c \chi \wedge R$ and $d\chi \wedge d^c \psi \wedge R$ are non-pluripolar measures, by Corollary~\ref{convergence}, we infer that
\[I_{2,k} \rightarrow \int_X u d\chi \wedge d^c \psi \wedge R \text{ as } k \rightarrow \infty.\]

Finally, we estimate $I_{3,k}$. By \eqref{convergencenonpluri1}, we have
\[\int_X u_l \chi \phi_k d d^c \psi_k \wedge R \rightarrow \int_X u_l \chi d d^c \psi \wedge R\]
as $k \rightarrow \infty$. Write 
\[\phi_k dd^c \psi_k \wedge R = \phi_k ( (C\omega + dd^c \max(\gamma_1,-k))\wedge R - (C\omega + dd^c \max(\gamma_2,-k))\wedge R ).\]
Since $\phi_k (C\omega + dd^c \max(\gamma_i,-k))\wedge R$ increases in $k$ to $(C\omega + dd^c\gamma_i)\wedge R$ for $i=1,2$, we can bound
\[\int_X |\chi||u_l-u| \phi_k (C\omega + dd^c \max(\gamma_i,-k))\wedge R \leq \int_X |\chi| |u_l-u|(C\omega + dd^c \gamma_i) \wedge R.\]
 Since $(C\omega + dd^c \gamma_i) \wedge R$ and $dd^c \psi \wedge R$ are non-pluripolar measures, by Lemmma~\ref{convergence}, we infer that
\[I_{3,k} \rightarrow \int_X u \chi dd^c \psi \wedge R \text{ as } k \rightarrow \infty.\]
Thus, \eqref{convergencenonpluri 1} follows from \eqref{convergencenonpluri 2} and the above estimates. The proof is complete.
\end{proof}

We have the following direct corollary.

\begin{corollary}\label{nonpluripolarmeasure}
    $du\wedge d^c \psi \wedge  R$ is a non-pluripolar measure.
\end{corollary}

\begin{proof}
Let $\chi$ be a smooth function on $X$. Let $(u_l)_l$ be the sequence in Lemma~\ref{regularize} associated to $u$. Since $\chi \phi_k \in W^{1,2}_*(X)$, by Lemma~\ref{generalconvergence}, we have
\[\int_X \chi \phi_k du_l \wedge  d^c \psi_k \wedge R \to \int_X \chi \phi_k du \wedge  d^c \psi_k \wedge R \text{ as } l \to \infty.\]
By Cauchy-Schwarz and Lemma~\ref{dv wedge dc v control}, we have
    \begin{align*} \Big|\int_X \chi \phi_k du_l \wedge  d^c \psi_k \wedge R \Big| &\leq \|\chi\|_{\mathscr{C}^0} \Big| \int_X  du_l\wedge d^c u_l \wedge R \Big|^{1/2} \Big|\int_X d\psi_k \wedge d^c \psi_k \wedge R \Big|^{1/2}\\
&\lesssim \|\chi\|_{\mathscr{C}^0}.
\end{align*}
Thus, by Lemma~\ref{convergencenonpluri}, $du\wedge d^c \psi \wedge  R$ is a current of order $0$.

We now argue similarly to Corollary~\ref{boundbycap}. Let $E$ be a pluripolar set and $E_l$ be a decreasing sequence of open sets such that $E = \cap_l E_l$. By Corollary~\ref{boundbycap}, there exists some continuous function (does not depend on $k$) $f:[0,\infty) \to [0,\infty)$ such that $f(0)=0$ and
\[|du \wedge d^c \psi_k \wedge R|(E_l) \leq f(\capa_\omega(E_l)) \text{ for all } k \text{ and } l.\]
By Lemma~\ref{convergencenonpluri}, we have 
\[|du\wedge d^c \psi \wedge R|(E_l) \leq \liminf_{k\to \infty}|du \wedge d^c \psi_k \wedge R|(E_l) \leq f(\capa_\omega(E_l)) \]
Thus, $|du\wedge d^c \psi \wedge R|(E)\leq f(\capa_\omega(E_l))$ for all $l$. Since $E$ is pluripolar, this implies $|du\wedge d^c \psi \wedge R|(E) = 0$ and the corollary follows.
\end{proof}

We now can prove the following important lemma.

\begin{lemma}\label{key}
    Let $v$ be a bounded function in $W^{1,2}_*(X)$. Then
    \[\int_X v d u \wedge d^c \psi \wedge R =\displaystyle \lim_{k\rightarrow \infty} \int_X v \phi_k du\wedge d^c\psi_k \wedge R .\]
\end{lemma}

\begin{proof}
    By Corollary~\ref{nonpluripolarmeasure}, all the integrals make sense. Let $(u_l)_l$ be the sequence in Lemma~\ref{regularize} associated to $u$. Let $E$ be a Borel subset of $X$. By Cauchy-Schwarz, we have
    \[\Big| \int_E \phi_kdu_l \wedge d^c \psi_k \wedge R \Big|\leq \Big|\int_E du_l \wedge d^c u_l \wedge R \Big|^{1/2} \Big|\int_E \phi_k d\psi_k \wedge d^c \psi_k \wedge R\Big|^{1/2} \lesssim \Big(\int_E d\psi\wedge d^c \psi \wedge R\Big)^{1/2}.\]
    Here, $\lesssim$ depending only on $u$ and $R$. Now, using Lemma~\ref{generalconvergence}, Lemma~\ref{convergencenonpluri}, and arguing similarly to Corollary~\ref{convergence}, the lemma follows.
\end{proof}

We now prove the main theorem of this section.

\begin{proof}[Proof of Theorem~\ref{integrationbypartdsh}]
    By Lemma~\ref{key}, we have
    \[\text{LHS} = \lim_{k\rightarrow \infty} \int_X  \rho(v_1,\ldots,v_m)\phi_k du\wedge d^c \psi_k \wedge R, \]
    \[-\text{RHS} =  \lim_{k\rightarrow \infty} \sum_{j=1}^m \int_X u\phi_k\partial_j\rho(v_1,\ldots,v_m) d v_j \wedge d^c \psi_k \wedge R  + \int_X u  \rho(v_1,\ldots,v_m) \phi_k  dd^c \psi_k \wedge R  .\]
    Let $\widetilde{\rho} (v_1,\ldots,v_m,\phi_k) = \rho(v_1,\ldots,v_m)\cdot \phi_k$. Applying Theorem~\ref{intbypastpsh} for $\widetilde{\rho}$, we get
    \begin{align*}
        -\text{LHS} & = \lim_{k\rightarrow \infty} \sum_{j=1}^m \int_X u\phi_k \partial_j\rho(v_1,\ldots,v_m) d v_j \wedge d^c \psi_k \wedge R  +  \int_X u  \rho(v_1,\ldots,v_m) \phi_k  dd^c \psi_k \wedge R \\
        &+ \int_X u \phi_k \rho(v_1,\ldots,v_m) d \phi_k \wedge d^c \psi_k \wedge R .
    \end{align*}
    We now estimate similarly to $I_{1,k}$ in the proof of Lemma~\ref{convergencenonpluri} to obtain that the last term goes to zero as $k\rightarrow \infty$. The proof is complete.
\end{proof}

\section{Cauchy-Schwarz inequality}\label{section CS}

We now prove a Cauchy-Schwarz-type inequality for complex Sobolev functions. This inequality will be used to obtain energy estimates in the next section.

\begin{theorem}\label{CSdsh}
Let $(X,\omega)$ be a compact K\"ahler manifold of dimension $n$. Let $u$ be a bounded function in $W^{1,2}_*(X)$, and $\psi$ be a bounded dsh function. Let $f,g$ be bounded Borel functions on $X$. Let $T$ be a closed positive $(1,1)$-current on $X$ such that $du\wedge d^c u \leq T$. Let $\theta_1,\ldots,\theta_{n-1}$ be closed $(1,1)$-forms in big cohomology classes. Let $\varphi_j \in \PSH(X,\theta_j)$ for $j=1,\ldots,n-1$. Let $R =  (\theta_1 + dd^c \varphi_1 )\wedge \cdots \wedge ( \theta_{n-1} + dd^c \varphi_{n-1} )$.
Then we have
    \begin{equation}\label{CS formula}
        \Big|\int_X fg du\wedge d^c \psi \wedge R \Big| \leq \Big(\int_{X} |f|^2 T\wedge R \Big)^{1/2} \Big( \int_{X} |g|^2 d\psi \wedge d^c \psi \wedge R\Big)^{1/2}.
    \end{equation}
    \end{theorem}

This theorem does not follow from the usual Cauchy-Schwarz inequality since $u$ and $\psi$ are not in the same space. 
We now comment on the proof of Theorem~\ref{CSdsh}. First, we note that when $R$ and $\psi$ are less singular, this theorem has been proved in \cite{Vu-diameter}. We recall this version here.

\begin{lemma}\cite[Lemma 2.13]{Vu-diameter}\label{CSbounded} Let $u,f,g$ be functions as in Theorem~\ref{CSdsh}. Suppose that $R$ is the product of $n-1$ closed positive $(1,1)$-currents with bounded potential, and $\psi$ is the difference of two bounded quasi-psh functions. Then we have
    \[
        \Big|\int_X fg du\wedge d^c \psi \wedge R \Big| \leq \Big(\int_{X} |f|^2 T\wedge R \Big)^{1/2} \Big( \int_{X} |g|^2 d\psi \wedge d^c \psi \wedge R\Big)^{1/2}.
    \]
\end{lemma}

In \cite{Vu-diameter}, the author proved Lemma~\ref{CSbounded} in the case when $\psi$ is a bounded quasi-psh function. However, the arguments presented there can be easily extended to this setting. The wedge product $T \wedge R$ in Lemma~\ref{CSbounded} is in the classical sense. However, since $du\wedge d^c \psi \wedge R$ puts no mass on pluripolar sets, we can switch to the non-pluripolar product.

\begin{corollary}\label{refineCSbounded}
Suppose $R$ and $\psi$ satisfy the conditions in Lemma~\ref{CSbounded}. Then we have
    \[ \Big| \int_X fg du\wedge d^c \psi \wedge R \Big| \leq \Big( \int_X |f|^2 \left \langle T \wedge R  \right \rangle \Big)^{1/2} \Big(\int_X |g|^2 d\psi \wedge d^c \psi \wedge R \Big)^{1/2}. \]
\end{corollary}

\begin{proof}
Let $A$ be the singular set of $T$. Since $du\wedge d^c \psi \wedge R$ puts no mass on pluripolar sets, by Lemma~\ref{CSbounded}, we have
\begin{align*} 
\Big| \int_X fg du\wedge d^c \psi \wedge R \Big| & = \Big| \int_X \mathbf{1}_{X\setminus A}fg du\wedge d^c \psi \wedge R\Big|\\&\leq \Big( \int_X |f|^2\mathbf{1}_{X\setminus A} T\wedge R \Big)^{1/2} \Big(\int_X |g|^2 d\psi \wedge d^c \psi \wedge R \Big)^{1/2} \\
&= \Big( \int_X |f|^2 \left \langle T\wedge R  \right \rangle \Big)^{1/2} \Big(\int_X |g|^2 d\psi \wedge d^c \psi \wedge R \Big)^{1/2}.
\end{align*}
Here, for the last line, we use \cite[Proposition 3.6. (ii)]{Viet-generalized-nonpluri}. The corollary follows.
\end{proof}

To prove Theorem~\ref{CSdsh}, we will follow a standard strategy. The first step is to generalize Corollary~\ref{refineCSbounded} to the case when $R$ is the product of closed positive $(1,1)$-currents of minimal singularities, and $\psi$ is the difference of two bounded quasi-psh functions (see Lemma~\ref{CSminimalsing}). Then we approximate $R$ by products of currents of minimal singularities (see Lemma~\ref{converge4}) and prove the Cauchy-Schwarz inequality in the case when $\psi$ is the difference of two bounded quasi-psh functions (see Lemma~\ref{CSquasipsh}). Finally, we approximate the bounded dsh function $\psi$ by $\psi_k$, which are differences of bounded quasi-psh functions (similar to Section~\ref{intbypart}), and obtain Theorem~\ref{CSdsh}. Let us now go into detail.

\begin{lemma}\label{CSminimalsing}
Suppose that $\varphi_j$ has minimal singularities for $j=1,\ldots,n$, and $\psi$ is the difference of two bounded quasi-psh functions on $X$. Then \eqref{CS formula} holds.
\end{lemma}

\begin{proof}

Denote by $\amp(\theta_j)$ the ample locus of $\{\theta_j\}$. Put $R_j : = \theta_j + dd^c \varphi_j$. Define $\Omega:= \bigcap_{j=1}^{n-1} \amp(\theta_j)$. Since $X\setminus \amp(\theta_j)$ is an analytic set for $j=1,\ldots,n$, we infer that $X\setminus \Omega$ is a closed pluripolar set.  Let $\{\mathcal{U}_\alpha\}_{\alpha=1}^m$ be a finite covering of $X$ by local charts. Then $\{\mathcal{V}_\alpha\}_{\alpha=1}^m$ is a finite covering of $\Omega$ by local charts, where $\mathcal{V}_\alpha = \mathcal{U}_\alpha \cap \Omega$. Let $\mathcal{V}_{\alpha,1} \Subset \mathcal{V}_{\alpha,2} \Subset \cdots \Subset \mathcal{V}_{\alpha}$ be an approximation of $\mathcal{V}_\alpha$ by relatively compact subsets. For $\beta \in \N$, let $\Omega_\beta = \bigcup_{\alpha=1}^m \mathcal{V}_{\alpha,\beta}$. Since $\mathcal{V}_{\alpha,\beta} \Subset \Omega$, we have $\Omega_\beta \Subset \Omega$. Hence, $R_j$ has bounded potential on $\Omega_\beta$.

By Corollary~\ref{boundbycap}, we have
\[\int_X fg du\wedge d^c \psi \wedge R = \int_{\Omega} fg du\wedge d^c \psi \wedge R = \lim_{\beta\rightarrow \infty} \int_{\Omega_\beta} fgdu\wedge d^c \psi \wedge R.\]
Recall that to prove Lemma~\ref{CSbounded}, one prove a local Cauchy-Schwarz inequality and then pass to the global version by a gluing argument (see the proof of \cite[Lemma 2.13]{Vu-diameter}). Thus, we can apply the Corollary~\ref{refineCSbounded} for $\mathcal{V}_{\alpha,\beta} \Subset \mathcal{V}_{\alpha,\beta+1}$ and glue them to get
\begin{align*}\int_{\Omega_\beta} fgdu\wedge d^c \psi \wedge R &\leq \Big(\int_{\Omega_{\beta+1}} |f|^2 T\wedge R \Big)^{1/2} \Big( \int_{\Omega_{\beta+1}} |g|^2 d\psi \wedge d^c \psi \wedge R\Big)^{1/2} \\ &\leq \Big(\int_{X} |f|^2 T\wedge R \Big)^{1/2} \Big( \int_{X} |g|^2 d\psi \wedge d^c \psi \wedge R\Big)^{1/2}.
\end{align*}
Letting $\beta\rightarrow \infty$, the proof is complete.
\end{proof}

 For $l>0$, define
\[R^{(l)} := (\theta_1 + dd^c ( \max (\varphi_1,V_{\theta_1}-l) ) ) \wedge \cdots \wedge (\theta_{n-1} + dd^c ( \max (\varphi_{n-1},V_{\theta_{n-1}}-l) ) ).\]
Let $\phi_l:= l^{-1} \max (\varphi_1+\cdots + \varphi_{n-1}, -l )+1$. It is standard that 
\begin{equation}\label{converge2}
    \phi_l R^{(l)}\wedge T  \text{ increases to } R\wedge T \text{ for every closed positive $(1,1)$-current } T.
\end{equation}

\begin{lemma}\label{converge3} $\phi_ldu\wedge d^c \psi \wedge R^{(l)} \rightarrow du\wedge d^c \psi \wedge R$ weakly as currents.
\end{lemma}

\begin{proof}
We argue similarly to Lemma~\ref{convergencenonpluri}. Let $\chi$ be a smooth function on $X$. By Theorem~\ref{intbypastpsh}, we have
    \begin{align}\label{converge3 1}\int_X \chi \phi_l du\wedge d^c \psi \wedge R^{(l)} = - \int_X \Big( & u \phi_l\chi dd^c \psi \wedge R^{(l)} + u \phi_l d\chi\wedge d^c \psi \wedge R^{(l)} \\ \nonumber  &+ u \chi d\phi_l \wedge d^c \psi \wedge R^{(l)} \Big).\end{align}
    Let $I_1^{(l)}, I_2^{(l)},$ and $I_3^{(l)}$ be the first, second, and third terms of the right-hand side of \eqref{converge3 1}.

    Assume that $\psi = \gamma_1 - \gamma_2$ where $\gamma_1,\gamma_2$ are bounded quasi-psh functions. Let $C$ be the positive constant such that $\gamma_1,\gamma_2$ are $C\omega$-psh functions. Applying \eqref{converge2} for $T = C\omega, T = C\omega + dd^c \gamma_1, T = C\omega+dd^c\gamma_2$, we get
\[ I_1^{(l)}=\int_X u \phi_l\chi dd^c \psi\wedge R^{(l)} \rightarrow \int_X u\chi dd^c \psi\wedge R \text{ as } l\rightarrow \infty.\]
Writing $d\chi \wedge d^c \psi$ as a linear combination of currents of the form $vT$ where $v$ is a bounded quasi-psh function and $T$ is a closed positive current. Applying \eqref{converge2} again, we get
\[I_2^{(l)} \rightarrow \int_X u d\chi \wedge d^c \psi \wedge R \text{ as } l\rightarrow \infty.\]

    Finally, we estimate $I_3^{(l)}$. By Lemma~\ref{CSminimalsing}, plurifine property, and Lemma~\ref{dv wedge dc v control}, we have
    \begin{align*}\Big| \int_X u \chi d\phi_l \wedge d^c \psi \wedge R^{(l)} \Big| &\leq \|u\chi\|_{L^\infty} \Big( \int_X d\phi_l \wedge d^c \phi_l \wedge R^{(l)} \Big)^{1/2} \Big(\int_X d\psi \wedge d^c \psi \wedge R^{(l)} \Big)^{1/2}  \\
    &\leq \|u\chi\|_{L^\infty} \Big( \int_X d\phi_l \wedge d^c \phi_l \wedge R \Big)^{1/2} \cdot (2 C\|\psi\|_{L^\infty} \|R^{(l)}\|_{X})^{1/2}.
    \end{align*}
    Arguing similarly to Lemma~\ref{convergencenonpluri}, we get $\int_X d\phi_l \wedge d^c \phi_l \wedge R \rightarrow 0$. Moreover, since we can bound uniformly $\|R^{(l)}\|_X$ by a constant depending only on cohomology classes, we infer that $I_3^{(l)} \to 0$ as $l\to \infty$. Thus, the lemma follows by \eqref{converge3 1} and Theorem~\ref{intbypastpsh}.
\end{proof}

\begin{corollary}\label{converge4}
    Let $f$ be a bounded Borel function on $X$, then
    \[\int_X f \phi_l du\wedge d^c \psi \wedge R^{(l)} \rightarrow \int_X f du\wedge d^c \psi \wedge R.\]
\end{corollary}

\begin{proof}
When $f$ is continuous, we can approximate $f$ in $\mathscr{C}^0$ topology by smooth functions and use Lemma~\ref{converge3} to complete the proof. For the general case, we apply Lusin's theorem for two measures $T\wedge R$ and $d\psi \wedge d^c \psi \wedge R$ to get a closed subset $E$ of $X$ satisfies
\[\int_{X\setminus E} T\wedge R < \varepsilon \quad \text{and} \quad \int_{X\setminus E} d\psi \wedge d^c \psi \wedge R < \varepsilon\] such that there is a continuous function $\widetilde{f} = f$ on $E$ and $\|\widetilde{f}\|_{L^\infty} \leq \|f\|_{L^\infty}$. By Lemma~\ref{CSminimalsing} and plurifine property, we have
\begin{align*}
    \Big| \int_X|f-\widetilde{f}| \phi_l du \wedge d^c \psi \wedge R^{(l)}  \Big| &\leq 2\|f\|_{L^\infty} \Big( \int_{X\setminus E} \phi_l T\wedge R^{(l)}\Big)^{1/2} \Big(\int_{X\setminus E} \phi_l d\psi \wedge d^c \psi \wedge R^{(l)}\Big)^{1/2} \\
    &\leq 2\|f\|_{L^\infty} \Big( \int_{X\setminus E}  T\wedge R\Big)^{1/2} \Big(\int_{X\setminus E} d\psi \wedge d^c \psi \wedge R\Big)^{1/2} \\
    &< 2\varepsilon \|f\|_{L^\infty}.
\end{align*}
Similarly, we have
\[\Big| \int_X|f-\widetilde{f}|  du \wedge d^c \psi \wedge R  \Big|\leq 2\varepsilon \|f\|_{L^\infty}\]
The result follows.
\end{proof}

By Lemma~\ref{CSminimalsing} and Corollary~\ref{converge4}, we get the following Cauchy-Schwarz inequality.

\begin{lemma}\label{CSquasipsh}
Suppose that $\psi$ is the difference of two bounded quasi-psh function on $X$. Then \eqref{CS formula} holds.
    \end{lemma}

We now can prove our main theorem of this section.

 \begin{proof}[Proof of Theorem~\ref{CSdsh}]
     Let $f$ be a bounded Borel function on $X$. Let $C$ be a positive constant such that $\psi = \gamma_1 - \gamma_2$ where $\gamma_1$ and $\gamma_2$ are $C\omega$-psh functions. We argue as in Lemma~\ref{converge4}. Applying Lusin's theorem for the measure $T\wedge R$, we get a closed subset $E$ of $X$ satisfies
\[\int_{X\setminus E} T\wedge R < \varepsilon \] 
such that there is a continuous function $\widetilde{f} = f$ on $E$ and $\|\widetilde{f}\|_{L^\infty} \leq \|f\|_{L^\infty}$. Let $\psi_k = \max(\gamma_1,-k)-\max(\gamma_2,-k)$. By Lemma~\ref{CSquasipsh} and Lemma~\ref{dv wedge dc v control}, we have
\begin{align*}
    \Big| \int_X|f-\widetilde{f}| \phi_k du \wedge d^c \psi_k \wedge R  \Big| &\leq 2\|f\|_{L^\infty} \Big( \int_{X\setminus E} T\wedge R\Big)^{1/2} \Big(\int_{X}  d\psi_k \wedge d^c \psi_k \wedge R\Big)^{1/2} \\
    &\leq 2\|f\|_{L^\infty} \varepsilon^{1/2} (2C\|\psi\|_{L^\infty} \|R\|_X)^{1/2}.
\end{align*}
Thus, we infer that
\[\int_X f du\wedge d^c \psi \wedge R = \lim_{k\rightarrow \infty}\int_X f \phi_k du\wedge d^c \psi_k \wedge R.\]
Applying Lemma~\ref{CSquasipsh} and letting $k\rightarrow \infty$, the proof is complete.
 \end{proof}

\section{Energy estimates}\label{energy estimate section}

In this section, we prove the energy estimates, generalizations of results in \cite[subsection 2.3]{Vu-diameter}. As in \cite{Vu-diameter}, these estimates play an important role in our geometric applications. Let $(X,\omega)$ be a compact K\"ahler manifold of dimension $n$.
Let $\theta$ be a closed $(1,1)$-form in a big cohomology class. Let $\gamma_1$ and $\gamma_2$ be $\theta$-psh functions of minimal singularities. Put $\psi = \gamma_1 - \gamma_2$. Then $\psi$ is a dsh function and $dd^c \psi = T - \eta$ where $T = \theta + dd^c \gamma_1, \eta = \theta + dd^c \gamma_2$ are closed positive $(1,1)$-currents. We have the following estimates.

\begin{proposition}\label{energy1}
    Let $M\geq 1$ and $p\geq 0$ be constants. Let $u\geq 1$ be a bounded function in $W^{1,2}_*(X)$. Suppose that $du\wedge d^c u \leq T = \eta + dd^c \psi$ and $0\leq \psi \leq M$. For every $0\leq m < n$, there holds
    \[\int_X u^p T^m \wedge \eta^{n-m} \leq 2M(p+1)^2 \int_X u^p T^{m+1} \wedge \eta^{n-m-1}.\]
    Consequently,
    \[\int_X u^p T^m \wedge \eta^{n-m} \leq 2^{n-m}(p+1)^{2(n-m)}M^{n-m} \int_X u^p T^n.\]
\end{proposition}

\begin{proposition}\label{energy2}
    Let $M\geq 1$ and $p\geq 2n$ be constants. Let $u\geq 0$ be a bounded function in $W^{1,2}_*(X)$. Suppose that $du\wedge d^c u \leq T = \eta + dd^c \psi$ and $0\leq \psi \leq M$. For every $0\leq m < n$, there holds
    \[\int_X u^p T^m \wedge \eta^{n-m} \leq 2M(p+1)^2  \Big( \int_X u^p T^{m+1} \wedge \eta^{n-m-1} + \int_X u^{p-2} T^{m+1} \wedge \eta^{n-m-1} \Big).\]
    Consequently,
    \[\int_X u^p T^m \wedge \eta^{n-m} \leq 2^{n-m}(p+1)^{2(n-m)}M^{n-m} \int_X (u^p+u^{p-2n}) T^n.\]
\end{proposition}

To prove these estimates, alongside Theorem~\ref{integrationbypartdsh} and Theorem~\ref{CSdsh}, we also need the following convergence lemma.

\begin{lemma}\label{key2}
    Let $v$ be a bounded Borel function. Let $R$ be a non-pluripolar product of $n-1$ closed positive $(1,1)$-currents. Let $\psi_k = \max(\gamma_1,-k) - \max(\gamma_2,-k)$ and $\phi_k = k^{-1} \max (\gamma_1+\gamma_2,-k)+1 $. Then 
    \[\int_X v d\psi \wedge d^c \psi \wedge R = \lim_{k\rightarrow \infty} \int_X v \phi_k d\psi_k \wedge d^c \psi \wedge R.\]
\end{lemma}

\begin{proof}
    By plurifine property, $\phi_k d\psi_k \wedge d^c \psi \wedge R = \phi_k d\psi_k \wedge d^c \psi_k \wedge R$. Moreover, since $\phi_k d\psi_k \wedge d^c \psi_k \wedge R$ increases to $ d\psi \wedge d^c \psi \wedge R $ as measures, the lemma follows.
\end{proof}

\begin{proof}[Proof of Proposition~\ref{energy1} and Proposition~\ref{energy2}]
    Let $R := T^k \wedge \eta^{n-k-1}$. We need to show
    \[\int_X u^p \eta \wedge R \leq 2M(p+1)^2 \int_X u^p T\wedge R.\]
Define
    \[J := \int_X u^p d\psi\wedge d^c \psi \wedge R.\]
    By Lemma~\ref{key2}, we have
    \[J = \lim_{k\rightarrow \infty} J_k \text{ where } J_k:= \int_X u^p \phi_k d\psi_k\wedge d^c\psi \wedge R\]
    Since $\psi_k$ is a bounded function in $W^{1,2}_*(X)$, by Theorem~\ref{integrationbypartdsh}, we get 
    \[J_k = -\int_X \psi_k u^p d\phi_k \wedge d^c \psi \wedge R - p\int_X  u^{p-1}\psi_k \phi_k du\wedge d^c\psi \wedge R - \int_X u^p \psi_k \phi_k dd^c \psi \wedge R.\]
    Letting $k\rightarrow \infty$. For the first term, by Lemma~\ref{key}, Lemma~\ref{dv wedge dc v control}, and Cauchy-Schwarz inequality, we have
    \begin{align*}\Big|\int_X \psi_k u^p d\phi_k \wedge d^c \psi \wedge R \Big| &= \lim_{l\to \infty}\Big| \int_X \psi_k u^p \phi_l d\phi_k \wedge d^c \psi_l \wedge R\Big| \\ &\leq \lim_{l\to \infty } M\|u\|^p_{L^\infty} \Big(\int_X d\phi_k \wedge d^c \phi_k \wedge R\Big)^{1/2} \Big(\int_X d\psi_l \wedge d^c \psi_l \wedge R\Big)^{1/2}  \\
    &\lesssim \Big(\int_X d\phi_k \wedge d^c \phi_k \wedge R\Big)^{1/2} .\end{align*}
    We now argue similarly to the estimate of term $I_{1,k}$ in Lemma~\ref{convergencenonpluri} to see that this term goes to $0$. The remaining term goes to
    \[-p\int_X  u^{p-1}\psi  du\wedge d^c\psi \wedge R - \int_X u^p \psi dd^c \psi \wedge R\]
    by Lebesgue's dominated theorem. Hence, we infer that
    \[J = -p\int_X  u^{p-1}\psi  du\wedge d^c\psi \wedge R - \int_X u^p \psi dd^c \psi \wedge R.\]
    We now use Theorem~\ref{CSdsh} and argue similarly to \cite[Proposition 2.15 and Proposition 2.16]{Vu-diameter} to complete the proof.
\end{proof}

\section{Sobolev inequality for complex Sobolev functions}\label{section Sobolev ineq for complex Sobolev}

In this section, we prove a \textit{singular} Sobolev inequality for a special class of complex Sobolev functions. In the proof of our main results, we will follow the strategy in \cite[Theorem 1.6]{Vu-diameter}. The main ingredients are energy estimates (which has been settled in Section~\ref{energy estimate section}) and Sobolev inequality associated to some auxiliary current. This current is the solution of a complex Monge-Amp\`ere equation in big cohomology class.  Uniform Sobolev inequalities for K\"ahler metrics in big cohomology classes has been proved in \cite{Nguyen_Vu_diamterforbigclass}. These inequalities required the current to be smooth outside an analytic set. However, it remains unknown whether such solution is smooth outside an analytic set. This is the reason why we need to remove this condition.

We first recall some notions in \cite{Nguyen_Vu_diamterforbigclass}. Let $(X,\omega)$ be a compact K\"ahler manifold of dimension $n$. Let $p>1,A\geq 1,B\geq 1$ be constants. We denote by $\mathcal{W}_{\bigclass}(X,\omega,p,A,B)$ the set of closed positive $(1,1)$-currents $T$ that satisfy the following conditions:
\begin{itemize}
        \item[(i)] $T$ is a smooth K\"ahler form on some open Zariski subset $U_T$ of $X$;
        \item[(ii)] $T$ belongs to a big cohomology class $\alpha$ and has minimal singularities;
        \item[(iii)]  There exists a closed $(1,1)$-form $\theta$ in $\alpha$ such that $\theta \leq A\omega$;
        \item[(iv)] $V_T^{-1} T^n = f \omega^n$ for some integrable function $f$ such that $\|f\|_{L^p(\omega^n)} \leq B$.
    \end{itemize}
We have the following Sobolev inequality for $T\in\mathcal{W}_{\bigclass}(X,\omega,p,A,B) $.
\begin{theorem}\label{uniformestimateLp} \cite[Theorem 1.2.(ii)]{Nguyen_Vu_diamterforbigclass}
Let $q\in (1,n/(n-1))$. Then for every $T\in \mathcal{W}_{\bigclass}(X,\omega,p,A,B)$ there exists a constant $C= C(\omega,n,p,A,B,q)$ such that
    \[\Big(\frac{1}{V_T}\int_{U_T} |u|^{2q}  T^n  \Big)^{1/q} \leq C \Big(\frac{1}{V_T} \int_{U_T}  du\wedge d^c u\wedge T^{n-1}  + \frac{1}{V_T} \int_{U_T} |u|^2  T^n  \Big)\]
    for every $u\in W_T^{1,2}(X)$.
\end{theorem}

We recall that $T$ has minimal singularities if we can write $T = \theta + dd^c \varphi$ for some $\theta$-psh function $\varphi$ such that $V_\theta - O(1) \le \varphi \le V_\theta +O(1)$. Recall also that $W^{1,2}_T(X)$ is the space of measurable functions $u$ such that both $u$ and $\nabla u$ are $L^2$ integrable with respect to smooth positive measure $T^n$ on $U_T$.

\begin{lemma}\label{sobolevW*}
    Let $u\geq 0$ be a bounded function in $W^{1,2}_*(X)$ and $du\wedge d^c u \leq S$ for some closed positive $(1,1)$-current $S$. Let $q\in (1,n/(n-1))$ and $r\geq 2$ be constants. Then for every $T\in \mathcal{W}_{\bigclass}(X,\omega,p,A,B)$ there exists a constant $C= C(\omega,n,p,A,B,q)$ such that
    \[\Big(\frac{1}{V_T} \int_X u^{rq} T^n \Big)^{1/q} \leq Cr^2 \Big(\frac{1}{V_T} \int_X u^{r-2} S\wedge T^{n-1} + \frac{1}{V_T}\int_X u^r T^n\Big).\]
\end{lemma}

\begin{proof}
We have $d(u^{r/2})\wedge d^c (u^{r/2}) = r^2 u^{r-2} du\wedge d^c u \leq r^2 u^{r-2} S$. Since $r\geq 2$, we get $u^{r/2} \in W_T(X)$. Applying Theorem~\ref{uniformestimateLp} for $u^{r/2}$, we infer that
    \begin{align*}\Big(\frac{1}{V_T} \int_{U_T} u^{rq} T^n \Big)^{1/q} &\leq C \Big(\frac{1}{V_T} \int_{U_T}  d (u^{r/2}) \wedge d^c (u^{r/2}) \wedge T^{n-1} + \frac{1}{V_T}\int_{U_T} u^{r} T^n\Big) \\
    &\leq Cr^2 \Big(\frac{1}{V_T} \int_{U_T} u^{r-2} du\wedge d^c u\wedge T^{n-1} + \frac{1}{V_T}\int_{U_T} u^r T^n\Big) \\
    &\leq Cr^2 \Big(\frac{1}{V_T} \int_{U_T} u^{r-2} S\wedge T^{n-1} + \frac{1}{V_T}\int_{U_T} u^r T^n\Big).
    \end{align*}
    Since $X\setminus U_T$ is a pluripolar set, the result follows.
\end{proof}

Denote by $\mathcal{W}^*_{\bigclass}(X,\omega,p,A,B)$ the set of closed positive $(1,1)$-currents $T$ that satisfy only conditions (ii), (iii), and (iv) in the definition of $\mathcal{W}_{\bigclass}(X,\omega,p,A,B)$. The following theorem has been proved in \cite[Section 7]{Nguyen_Vu_diamterforbigclass}.

\begin{theorem}\label{approximate T by good ones}
    Given $T \in \mathcal{W}^*_{\bigclass}(X,\omega,p,A,B)$. Let $q\in (1,n/(n-1))$. Then there exists a sequence of closed positive $(1,1)$-currents $(T_k)_k$ such that the potential of $T_k$ converges to the potential of $T$ in capacity. Moreover, $T_k$ is a smooth K\"ahler form outside an analytic set, and there exists a constant $C = C(\omega,n,p,A,B,q)$ such that
    \begin{equation}\label{eq sobolev ineq Tk}\Big(\frac{1}{V_{T_k}}\int_{U_{T_k}} |u|^{2q}  T_k^n  \Big)^{1/q} \leq C \Big(\frac{1}{V_{T_k}} \int_{U_{T_k}}  du\wedge d^c u\wedge T_k^{n-1}  + \frac{1}{V_{T_k}} \int_{U_{T_k}} |u|^2  T_k^n  \Big)\end{equation}
    for every $u\in W_{T_k}^{1,2}(X)$.
\end{theorem}

\begin{proof}
    We briefly explain why one can prove this theorem by using arguments in \cite[Section 7]{Nguyen_Vu_diamterforbigclass}. Let $T \in \mathcal{W}^*_{\bigclass}(X,\omega,p,A,B)$. Then, by using Demailly's regularization theorem, one can approximate $T$ by a sequence $(T_k')_k$ such that $T_k'$ is a smooth K\"ahler form outside an analytic set and has analytic singularities. One can modify $T_k'$ (by taking blow-ups, solving the complex Monge-Amp\`ere equation in the prescribed singularity setting, and pushing down) to obtain $T_k$ that satisfies~\eqref{eq sobolev ineq Tk}. Condition (i) is only needed when we need to pass the Sobolev inequality to $T$ (to have a good definition of $du\wedge d^c u \wedge T$).
\end{proof}

We now prove the main result of this section.

\begin{theorem}\label{sobolevW* 2}
    Let $u\geq 0$ be a bounded function in $W^{1,2}_*(X)$ and $du\wedge d^c u \leq S$ for some closed positive $(1,1)$-current $S$. Suppose that $u$ is continuous outside a closed pluripolar set. Let $q\in (1,n/(n-1))$ and $r\geq 2$ be constants. Then for every $T\in \mathcal{W}^*_{\bigclass}(X,\omega,p,A,B)$ there exists a constant $C= C(\omega,n,p,A,B,q)$ such that
    \[\Big(\frac{1}{V_T} \int_X u^{rq} T^n \Big)^{1/q} \leq Cr^2 \Big(\frac{1}{V_T} \int_X u^{r-2} S\wedge T^{n-1} + \frac{1}{V_T}\int_X u^r T^n\Big).\]
\end{theorem}

\begin{proof}
    Assume that $u$ is continuous on an open set $\Omega$ of $X$ such that $X\setminus \Omega$ is a pluripolar set. Let $(\chi_l)_l$ be a sequence of smooth cut-off functions such that $\chi_l$ has compact support on $\Omega$ and $\chi_l$ increases to $\mathbf{1}_{X\setminus \Omega}$. We have
    \begin{align}\label{d chi l u r dc chi l u r}
    d (\chi_l u^{r/2} )\wedge d^c (\chi_l u^{r/2}) &= u^{r} d\chi_l \wedge d^c \chi_l + \chi_l^2 d(u^{r/2}) \wedge d^c (u^{r/2}) \\ \nonumber &+ u^{r/2} \chi_l (d\chi_l \wedge d^c(u^{r/2}) + d(u^{r/2})\wedge d^c \chi_l).
\end{align} 

Let $T_k$ be the sequence in Theorem~\ref{approximate T by good ones} associated to $T$. Using Sobolev inequality in Theorem~\ref{approximate T by good ones} and arguing similar to proof of Lemma~\ref{sobolevW*}, we get
\[\Big(\frac{1}{V_{T_k}}\int_{X} (\chi_l u^{r/2})^{2q}  T_k^n  \Big)^{1/q} \leq C \Big(\frac{1}{V_{T_k}} \int_{U_{T_k}}  d(\chi_l u^{r/2})\wedge d^c (\chi_l u^{r/2})\wedge T_k^{n-1}  + \frac{1}{V_{T_k}} \int_{X} \chi_l^2u^r  T_k^n  \Big).\]
We note that $\chi_l u^{r/2}$ is continuous on $X$.

We want to control term $d(\chi_l u^{r/2})\wedge d^c (\chi_l u^{r/2})\wedge T_k^{n-1}$. By formula~\eqref{d chi l u r dc chi l u r}, we need to control four terms: 

The first term is $u^r d\chi_l \wedge d^c \chi_l \wedge T_k^{n-1}$. We can estimate 
\begin{equation}\label{estimate first term}\int_X u^{r} d\chi_l \wedge d^c \chi_l \wedge T_k^{n-1} \leq \|u\|^r_{L^\infty} \int_X d\chi_l \wedge d^c \chi_l \wedge T_k^{n-1}.\end{equation}

Next term is $ \chi_l^2 d(u^{r/2}) \wedge d^c (u^{r/2}) \wedge T_k^{n-1}$. Note that $d(u^{r/2}) \wedge d^c (u^{r/2}) \leq r^2 u^{r-2} S$, we can bound
\begin{equation}\label{estimate second term}
    \int_{U_{T_k}} \chi_l^2 d(u^{r/2}) \wedge d^c (u^{r/2}) \wedge T_k^{n-1} \leq \int_X \chi_l^2 r^2 u^{r-2} S \wedge T_k^{n-1}.
\end{equation}

The remaining terms are $u^{r/2} \chi_l d\chi_l \wedge d^c (u^{r/2}) \wedge T^{n-1}_k$ and $u^{r/2} \chi_l d(u^{r/2})\wedge d^c \chi_l \wedge T_k^{n-1}$. We will use the same bound for both term. By Theorem~\ref{CSdsh}, we have

\begin{align}\label{estimate third term} \int_{U_{T_k}}u^{r/2} \chi_l d\chi_l \wedge d^c &(u^{r/2}) \wedge T^{n-1}_k \leq \|u\|_{L^\infty}^{r/2} \int_X d\chi_l \wedge d^c(u^{r/2}) \wedge T_k^{n-1} \\ \nonumber
&\leq \|u\|_{L^{\infty}}^{r-1} r \Big(\int_X d\chi_l \wedge d^c \chi_l \wedge T_k^{n-1}\Big)^{1/2} \Big(\int_X S\wedge T_k^{n-1}\Big)^{1/2}.
\end{align}
Here, we use non-pluripolar product in the first inequality. 

By \cite[Theorem 5.4]{Nguyen_Vu_diamterforbigclass}, $T_k^n \rightarrow T^n$ and $S\wedge T_k^{n-1} \rightarrow S\wedge T^{n-1}$ weakly as measures. Combining \eqref{estimate first term}, \eqref{estimate second term}, \eqref{estimate third term}, and letting $k\rightarrow \infty$ give us the bound
\begin{align*}\Big(\frac{1}{V_T} \int_X (\chi_l u^{r/2})^{2q}& T^n \Big)^{1/q} \leq \frac{C}{V_T}\Big( \|u\|^r_{L^\infty} \int_X d\chi_l \wedge d^c \chi_l \wedge T^{n-1} + \int_X \chi_l^2 r^2 u^{r-2} S \wedge T^{n-1} \\
&+2 \|u\|_{L^{\infty}}^{r-1} r \Big(\int_X d\chi_l \wedge d^c \chi_l \wedge T^{n-1}\Big)^{1/2} \Big(\int_X S\wedge T^{n-1}\Big)^{1/2} 
+ \int_X \chi_l^2 u^r T^n\Big).
\end{align*}
Since $T^{n-1}$ puts no mass on $X\setminus \Omega$, it is standard that 
\[\int_X d\chi_l \wedge d^c \chi_l \wedge T^{n-1} \to 0 \]
as $l\rightarrow \infty$ (see e.g. proof of Lemma~\ref{convergencenonpluri}). Letting $l \rightarrow \infty$ give us the desired estimate. The proof is complete.
\end{proof}

\begin{remark}
    We note that to prove Theorem~\ref{mainresult}, it is enough to obtain Theorem~\ref{sobolevW* 2} for the continuous function $u$, since we consider the ball lying inside the smooth locus $U_T$ of $T$. We choose to prove this version to show that one can run similar machinery as in \cite{Vu-diameter} to obtain results in diameter estimates for K\"ahler metrics in big cohomology classes. We also expect that Theorem~\ref{sobolevW* 2} holds for a general bounded function $u\in W^{1,2}_*(X)$.
\end{remark}

\section{Proof of main results}\label{proof of main result}

We prove our main results in this section. Our proof is similar to that of \cite{Vu-diameter}, with new ingredients being energy estimates in Section \ref{energy estimate section} and new Sobolev inequalities in Section~\ref{section Sobolev ineq for complex Sobolev}.

\begin{proof}[Proof of Theorem~\ref{mainresult}] Let $T$ be a singular metrics in $\mathcal{M}_{\bigclass}(X,A)$ that satisfies~\eqref{ine-dkLplocal}. Then, we can write $V_T^{-1}T^n = f \omega_X^n$ on $U_T$ where $f$ is a smooth function on $U_T$. Let $\chi$ be the cut-off function such that $0\leq \chi\leq 1$, $\chi$ has support on $B_{\d_T}(x_0,2R_0/3)$ and $\chi \equiv 1$ on $B_{\d_T}(x_0,R_0/2)$. Put $f_1 = \chi f\geq 0$. We note that
\[\int_X |f_1|^{p_0}\omega_X^n = \int_{B_{\d_T}(x_0,2R_0/3)} |\chi f|^{p_0} \omega_X^n \leq \int_{B_{\d_T}(x_0,R_0)} |f|^{p_0} \omega_X^n
 \leq K.\]
Since $T\in \mathcal{M}_{\bigclass}(X,A)$, we can write $T = \theta + dd^c \varphi$ where $\theta$ is a closed smooth form such that $\|\varphi - V_\theta\|_{L^\infty} \leq A$. Let $\widetilde{T} := \theta + dd^c \widetilde{\varphi}$ where $\widetilde{\varphi}$ satisfies 
\[V_T^{-1}(\theta + dd^c \widetilde{\varphi})^n = f_1 \omega_X^n + V_{\omega_X}^{-1} \Big(1 - \int_X f_1 \omega_X^n \Big) \omega_X^n.\]

Let $\phi$ be a $\theta$-psh function with $\sup_X \phi = 0$. Note that from the construction, we can pick a constant $B$ depending only on $A$ such that $\theta \leq B \omega_X$. This implies $\phi$ is also a $B\omega_X$-psh function. So, there exists $\alpha = \alpha(A)$ and $B_1 = B_1 (\alpha)$ such that: for every $\theta$-psh function $\phi$ with $\sup_X \phi = 0$,
\[\int_X e^{\alpha|\phi|} \omega_X^n \leq B_1.\]
Put $g = f_1 + V_{\omega_X}^{-1} (1- \int_X f_1 \omega_X^n)$, we can bound $\int_X g^{p_0} \omega_X^n \leq B_2$ where $B_2$ is a constant depending on $p_0,K$ and $\omega_X$. Hence, by \cite[Theorem 1.9]{DiNezzaGG}, there exists a constant $M = M(\omega_X,A,p_0,K)$ such that $\|\widetilde{\varphi}-V_\theta\|_{L^\infty}\leq M$. We have shown that $T - \widetilde{T} = dd^c \psi$ where $\psi$ is a dsh function such that $\|\psi\|_{L^\infty} \leq M$, where $M$ is some uniform constant.

Let $\rho: \R \rightarrow \R_{\geq 0}$ be a smooth function such that $0 \leq \rho \leq 1$, $\rho = 1 $ on $[-1/2,1/2]$ and $\rho = 0$ outside $[-2/3,2/3]$. Let $u = \d_T(x_0,\cdot)$ and $u_r = \rho(u/r)$ for $r\leq R_0/2$, then $u_r$ has support in $B_{\d_T}(x_0,R_0/3)$. We have
\begin{equation}\label{du_rd^cu_r}
du_r \wedge d^c u_r \leq D r^{-2} T
\end{equation}
where $D$ is a constant depending only on $\rho$ (for a proof of this inequality, see e.g. \cite[Lemma 2.1]{Vu-log-diameter}). Let $p\geq 2n+2$ be a constant, by Proposition~\ref{energy2},
\begin{align*}\int_X u_r^{p-2} T\wedge \widetilde{T}^{n-1} &= (Dr^{-2})^{-n}\int_X u_r^{p-2} (Dr^{-2}T)\wedge (Dr^{-2}\widetilde{T})^{n-1} \\
&\leq (Dr^{-2})^{-1} 2^{n-1}(p-1)^{2(n-1)}M^{n-1} \int_X (u_r^{p-2} + u_r^{p-2-2n})(Dr^{-2}T)^n
\\
&\leq (Dr^{-2})^{n-1}2^{n}M^{n-1}(p-1)^{2(n-1)} \int_X u_r^{p-2-2n} T^n.
\end{align*}
Here, we use $0\leq u_r \leq 1$ for the last line. By Theorem~\ref{sobolevW* 2}, for $q\in (1,n/(n-1))$, we get
\begin{align*}
\Big(\frac{1}{V_T}\int_X u_r^{pq} \widetilde{T}^n \Big)^{1/q} &\leq Cp^2\Big(\frac{1}{V_T}\int_X u_r^{p-2} (Dr^{-2}T)\wedge \widetilde{T}^{n-1} + \frac{1}{V_T}\int_X u_r^p \widetilde{T}^{n}\Big) \\
&\leq Cp^{2n} (Dr^{-2})^{n} 2^{n-1} \Big(\frac{1}{V_T}\int_X u_r^{p-2-2n}T^n\Big) + Cp^2\Big(\frac{1}{V_T}\int_X u_r^p \widetilde{T}^n \Big) \\
&\leq Cp^{2n}(Dr^{-2})^n2^n \Big( \frac{1}{V_T}\int_X u_r^{p-2-2n}\widetilde{T}^n \Big).
\end{align*}
Here, we use the fact $T^n \leq \widetilde{T}^n$ on $B_{\d_T}(x_0,R_0/2)$ and again $0\leq u_r \leq 1$ for the last line. Note also that when $r\rightarrow 0$, $r^{-2n} \rightarrow \infty$.

Putting $p= 4n+3$, 
%and use again the fact $T^n \leq \widetilde{T}^n$ on $B_{\d_T}(x_0,R_0/2)$, 
we get
\[\Big(\frac{1}{V_T}\int_X u_r^{(4n+3)q} T^n \Big)^{1/q} \leq C(4n+3)^{2n}(Dr^{-2})^n2^n \Big( \frac{1}{V_T}\int_X u_r^{2n+1}\widetilde{T}^n \Big).\]
Applying Proposition~\ref{energy2} for the right-hand side, we get
\begin{align*}\int_X u_r^{2n+1} \widetilde{T}^{n}
&\leq  2^{n}(2n+2)^{2n}M^{n} \int_X (u_r^{2n+1} + u_r)(Dr^{-2}T)^n
\\
&\leq (Dr^{-2})^{n}2^{n+1}(2n+2)^{2n}M^n \int_X u_r T^n.
\end{align*}
This deduces
\[\Big(\frac{1}{V_T}\int_X u_r^{(4n+3)q} T^n \Big)^{1/q} \leq Cr^{-4n} \Big(\frac{1}{V_T}\int_X u_r T^n \Big)\]
for some constant $C = C(\omega_X,n,A,p_0,K)$. Now, by the definition of $u_r$, we infer that
\[\Big( \frac{\vol_T(B_{\d_T}(x_0,r/2))}{V_T}\Big)^{1/q} \leq C \Big( \frac{\vol_T(B_{\d_T}(x_0,r))}{r^{4n}V_T}\Big).\]
This is equivalent with
\[\Big( \frac{\vol_T(B_{\d_T}(x_0,r/2))}{(r/2)^aV_T}\Big)^{1/q} \leq C \Big( \frac{\vol_T(B_{\d_T}(x_0,r))}{r^{a}V_T}\Big),\]
where $a= 4n q/(q-1) > 0$. Applying this inequality for $r_m = 2^{-m}r$ with $m=0,1,\ldots,l$, we obtain
\begin{equation}\label{compareradius}\Big( \frac{\vol_T(B_{\d_T}(x_0,2^{-l-1}r))}{(2^{-l-1}r)^aV_T}\Big)^{1/q^{l+1}} \leq C^{\sum_{m=0}^l q^{-m}} \Big( \frac{\vol_T(B_{\d_T}(x_0,r))}{r^{a}V_T}\Big).\end{equation}
We note that when $l\rightarrow \infty$, the ball $B_{\d_T}(x_0,2^{-l-1}r)$ is approximately Euclidean. This means
\[\lim_{l\rightarrow \infty} \frac{\vol_T(B_{\d_T}(x_0,2^{-l-1}r))}{(2^{-l-1}r)^{2n}V_T} = \frac{c_n}{V_T},\]
where $c_n$ is the volume of the unit ball in $\C^n$. From this, we have
\begin{align*}&\lim_{l\rightarrow \infty}\Big( \frac{\vol_T(B_{\d_T}(x_0,2^{-l-1}r))}{(2^{-l-1}r)^aV_T}\Big)^{1/q^{l+1}}\\ &= \lim_{l\rightarrow \infty} \Big( \frac{\vol_T(B_{\d_T}(x_0,2^{-l-1}r))}{(2^{-l-1}r)^{2n}V_T}\Big)^{1/q^{l+1}} (2^{-l-1}r)^{(2n-a)/q^{l+1}} =1.
\end{align*}
Moreover, we have $\lim\limits_{l\rightarrow \infty}C^{\sum_{m=0}^l q^{-m}} = C^{q/(q-1)} $. Now, letting $l\rightarrow \infty$ in \eqref{compareradius} give us
\[\Big( \frac{\vol_T(B_{\d_T}(x_0,r))}{r^{a}V_T}\Big) \geq C^{-q/(q-1)}.\]
The proof is complete.
\end{proof}

\begin{proof}[Proof of Corollary~\ref{corollaryofmainresult}]
    Let $C$ be the constant in the Theorem~\ref{mainresult}. We note that $C$ is not depending on the choice of $x_0$ and $R_0$. Assume that $k = C^{-1}$ is an integer (we can let $C$ smaller if necessarily).
    Let $x\in U$, we will show that $\d_T(x,\partial U) \leq 6(k + 1)$. Suppose otherwise $\d_T(x,\partial U ) > 3 (k + 1)$. We then pick a sequence of points $(x_i)_{i=0}^k$ where $x_0 = x$, $x_i \in B_{\d_T}(x_0,3 i + 1) \setminus B_{\d_T} (x_0, 3i)$ for $i = 1 ,\ldots,k$. The balls $B_{\d_T}(x_i,1)$ are disjoints and contained in $U$. By Theorem~\ref{mainresult}, the volume of $B_{\d_T}(x_i,1)$ is greater than $V_T/k$. Thus, we get
    \[V_T \geq \vol_T(U) \geq  \sum_{i=0}^k\vol_T(B_{\d_T}(x_i,1)) \geq  \frac{k+1}{k}V_T.\]
    This is a contradiction. The result follows.
\end{proof}

\bibliography{biblio_family_MA,biblio_Viet_papers,bib-kahlerRicci-flow}
\bibliographystyle{alpha}

\bigskip

\noindent
\Addresses
\end{document}